\documentclass{article}
\usepackage{amsmath,amssymb,amsthm,mathtools,mathdots,nicefrac,tikz,enumitem,url, hyperref}

\usepackage[a4paper,top=2.5cm,bottom=2cm,left=2cm,right=2cm,marginparwidth=1.75cm]{geometry}

\usepackage{algorithm}
\usepackage[noend]{algpseudocode}
\usepackage{multirow}

\hypersetup{% 
  colorlinks=true,
  allcolors=blue,
  linkcolor=blue,
  linkbordercolor={0 0 1}
}
\theoremstyle{definition} % non-italic theorems
\newtheorem{theorem}{Theorem}[section]
\newtheorem{corollary}[theorem]{Corollary}
\newtheorem{lemma}[theorem]{Lemma}
\newtheorem{proposition}[theorem]{Proposition}
\newtheorem{definition}[theorem]{Definition}
\newtheorem{example}[theorem]{Example}
\newtheorem{remark}[theorem]{Remark}
\newtheorem{conjecture}[theorem]{Conjecture}
\newtheorem{question}[theorem]{Question}

\newcommand{\scramble}{\mathcal{S}}
\makeatletter
\newcommand{\vast}{\bBigg@{3}}
\makeatother

\DeclareMathOperator{\val}{val}
\DeclareMathOperator{\tw}{tw}
\DeclareMathOperator{\gon}{gon}

\DeclareMathOperator{\scw}{scw}

\DeclareMathOperator{\adh}{adh}

\DeclareMathOperator{\lw}{lw}
\DeclareMathOperator{\bw}{bw}

\DeclareMathOperator{\cart}{cart}
\DeclareMathOperator{\dsn}{dsn}
\DeclareMathOperator{\sn}{sn}
\DeclareMathOperator{\bn}{bn}
\DeclareMathOperator{\vcon}{vcon}

\usepackage{mathrsfs} 

\title{On the size and complexity of scrambles}
\author{Seamus Connor, Steven DiSilvio, Sasha Kononova, Ralph Morrison, and Krish Singal}
\date{}

\begin{document}

\maketitle

\begin{changemargin}{1cm}{1cm} 

\begin{center}
    \textbf{Abstract}
\end{center}

The scramble number of a graph, a natural generalization of bramble number, is an invariant recently developed to study chip-firing games and graph gonality. We introduce the carton number of a graph, defined to be the minimum size of a maximum order scramble, to study the computational complexity of scramble number. We show that there exist graphs with carton number exponential in the size of the graph, proving that scrambles cannot serve as an NP certificate. 
We present families of graphs whose scramble number and gonality can be constant-factor approximated in polynomial time and show that the disjoint version of scramble number is fixed parameter tractable. Lastly, we find that vertex congestion is an upper bound on screewidth and thus scramble number, leading to a new proof of the best known bound on the treewidth of line graphs and a bound on the scramble number of planar graphs with bounded degree.

\end{changemargin}

\section{Introduction}

The concept of gonality, initially from the theory of algebraic curves, has analogues in several other fields of math. One such analogue appears in the context of chip-firing games on graphs, where the gonality of a graph is the smallest degree of a positive rank divisor on that graph. In other words, it is the smallest number of chips that can be distributed among the graph's vertices such that debt can be eliminated, using chip-firing moves, after $-1$ chips are placed on any vertex. Here a chip-firing move is made by choosing a vertex and moving one chip from the vertex to its neighbors along each edge.

To prove the gonality of a graph is equal to $k$, one has to argue $k$ serves as both an upper and a lower bound.  To argue the upper bound, it suffices to provide a positive rank divisor of this graph with degree $k$. Verifying the divisor in fact has positive rank can be done in polynomial time via repeated applications of Dhar's burning algorithm \cite{dhar}. On the other hand, proving lower bounds on gonality is often much harder, as the brute-force approach of showing each of $\Omega(n^{k-1})$-many effective divisors of degree $k-1$ have rank less than $1$ is not computationally tractable in general.

A powerful lower bound on the gonality of a graph $G$ is the much-studied invariant of treewidth, $\tw(G)$, as proved in \cite{debruyn2014treewidth}.  This was improved upon in \cite{new_lower_bound}, which introduced the scramble number of a graph, denoted $\sn(G)$. A scramble on a graph $G$ is a set of connected subgraphs (called \emph{eggs}) of $G$. The order of a scramble is a quantity associated with the hitting number and connectivity between the eggs. The scramble number of a graph $G$ is the maximum order of a scramble on $G$. From \cite[Theorem 1.1]{new_lower_bound}, we have
\begin{align*}
    \tw(G) \leq \sn(G) \leq \gon(G).
\end{align*}

Echavarria et al.\ showed in 2021 that $\mathrm{sn}(G)$ is NP-hard to compute \cite{echavarria2021scramble}. However, it is not known whether lower-bounding (or upper-bounding) scramble number is in NP. While a maximum-order scramble would be a natural NP-certificate candidate for a lower bound, the size of a scramble (as measured by the number of eggs) is potentially exponential in the number of vertices of $G$. Many different scrambles may have order equal to scramble number, so it is natural to ask about the size of the smallest such scramble. To study this certificate, we introduce the \emph{carton number of a graph}, $\cart(G)$, defined as the minimum size of a maximum order scramble. 

Our first result establishes a general lower bound on the carton number of graphs whose maximum degree is less than their scramble number.  Below $\Delta(G)$ denotes the maximum degree of a graph, and $V(G)$ its set of vertices.

\begin{theorem} \label{general-lower-bound}
    Let $G$ be a graph such that $\Delta(G) < \sn(G)$. Then, $\cart(G) \geq 3 \sn(G) - |V(G)|$. 
\end{theorem}

 Our second  result confirms our suspicion regarding the invalidity of scrambles as NP-certificates. 

\begin{theorem} \label{expander-graphs} 
    For any $d\geq 3$, there exists a family of $d$-regular graphs $(G_n)_{n \geq 1}$ such that
    \begin{enumerate}
        \item $\sn(G) = \Omega(n)$, and
        \item $\cart(G) = 2^{\Omega(\sqrt{n})}$.
    \end{enumerate}
\end{theorem}

% To apply this result, we show that there exists a family of $d$-regular graphs with carton number $2^{\Omega(\sqrt{n})}$, implying that $\cart(G)$ can grow exponentially in $n=|V(G)|$. 
This result definitively invalidates scrambles as NP-certificates in the general case. We also consider the related invariant of disjoint scramble number, denoted $\dsn(G)$, introduced in \cite{sn_two}.  This is defined identically to $\sn(G)$, except that we restrict to scrambles where the eggs are non-overlapping. Deciding whether $\dsn(G) \geq k$ is in NP since any disjoint scramble necessarily has polynomial size and its order can be computed in polynomial time; however, it is unknown whether computing $\dsn(G)$ is NP-hard.  We prove the following result regarding its computational complexity.

\begin{theorem} \label{dsn-fpt}
    Deciding whether $\dsn(G) \geq k$ is fixed parameter tractable when parametrized by $\tw(G)$ and $k$. 
\end{theorem} 

From here we consider our computational problems from the perspective of approximation.  Although gonality is known to be APX-Hard for general graphs \cite{gijswijt2019computing}, it may admit approximation algorithms for specific graph families. In particular, a deterministic $\rho$-approximation algorithm $A$ for a minimization objective function $f$ guarantees
\begin{align*}
    OPT \leq A(x) \leq \rho \cdot OPT
\end{align*}
where $OPT$ denotes the optimal cost and $\rho > 1$. Similarly, a deterministic $\rho$-approximation algorithm $A$ for a maximization objective function $f$ guarantees
\begin{align*}
    OPT / \rho \leq A(x) \leq OPT
\end{align*}
where $OPT$ denotes the optimal cost and $\rho > 1$. We characterize various graph families whose scramble number and gonality can be constant-factor approximated in polynomial time. In doing so, we prove the following two theorems; in these $\alpha_k(G)$ denotes the \emph{$k$-component independence number}; $\textrm{girth}(G)$ denotes the girth of the graph $G$; $\delta(G)$ denotes the minimum degree of a vertex of $G$; and $\xi_3(G)$ denotes the minimum outdegree of any $3$-vertex-connected subgraph of $G$.  (We refer the reader to Section \ref{preliminaries} for full definitions.)

\begin{theorem} \label{k-hitting-set}
    Let $G$ be a simple graph on $n$ vertices. The quantity $n - \alpha_{k-1}(G)$ can be $k$-approximated in polynomial time. 
\end{theorem}

 This theorem can be seen as a generalization of Gavril's algorithm that $2$-approximates the minimum vertex cover of a simple graph. To the best of our knowledge, the result was previously unknown. 

\begin{theorem} \label{gon-sn-approximation}
    Let $G$ be a simple graph on $n$ vertices. There exists a polynomial time algorithm that approximates both $\gon(G)$ and $\sn(G)$ within a constant factor for each of the following graph families. 

    \begin{enumerate}[label = (\arabic*)]
        \item $\mathrm{girth}(G) \geq 4$, $\delta(G) \geq 3$, $\xi_3(G) \geq n+1$ ($3$-approximation)
        \item $\mathrm{girth}(G) \geq 4$, $|V(G)| \geq 6$, $\delta(G) \geq \frac{n}{3}+1$ ($3$-approximation)
        \item $\mathrm{girth}(G) \geq 5$, $|V(G)| \geq 8$, $\delta(G) \geq \frac{1}{2}(\lfloor \frac{n}{2} \rfloor + 4)$ ($4$-approximation)
        \item For any two adjacent vertices $u$ and $v$, $\val(u) + \val(v) \geq n$ and for any two non-adjacent vertices $u$ and $v$, $\val(u) + \val(v) \geq n+1$ ($2$-approximation)
    \end{enumerate}
\end{theorem}

Our last result provides an upper bound on scramble number in terms of maximum degree and treewidth.  
\begin{theorem}\label{theorem:sn_at_most_tw_times_delta}
For any graph $G$ of maximum degree $\Delta(G)$, we have that $\sn(G)\leq (\tw(G)+1)\Delta(G)$.
\end{theorem}

It follows that planar graphs of bounded degree have scramble number at most $O(\sqrt{n})$.  Another consequence of this theorem is a novel proof of a result relating the treewidth of a graph $G$ to the treewidth of its corresponding line graph $L(G)$, namely that
\[\tw(L(G)) \geq \tw(G) - 1.\]  
To the best of our knowledge, this is the strongest known such bound, first proven as \cite[Proposition 2.3]{treewidth-of-line-graphs}. 

% In practice, many different scrambles on a given graph may have order equal to scramble number, so it is natural to ask for a scramble that is the most manageable to work with. In this paper, we introduce the \textit{carton number} of a graph as the minimum number of eggs in a scramble attaining scramble number. Carton number is thus important for computational results around scramble number: it provides a limit on which scrambles one must compute the order of to find the scramble number. This is especially important given that scramble number, and even the order of a scramble, are hard to compute (cite BLAH).  

% definitions: gonality, scrambles/scramble number, carton number

Our paper is organized as follows. In Section \ref{preliminaries}, we outline preliminary background relevant to our paper as well as some basic properties of carton number following from its definition. In Section \ref{carton-props-section}, we outline some useful properties of carton number that will aid in demonstrating lower bounds, as well as properties regarding its behavior under various graph operations.  In Section \ref{bounds-on-carton-number} we 
show a series of lower bounds on carton number culminating in Theorems \ref{general-lower-bound} and \ref{expander-graphs}. In Section \ref{properties-disjoint-scramble-number} we establish its relationship to other graph invariants and use this to pinpoint the carton number of various graph families. In Section \ref{complexity-sn-variants}, we show Theorem \ref{dsn-fpt} and prove approximability Theorems \ref{k-hitting-set} and \ref{gon-sn-approximation}. Finally, in Section \ref{congestion-section} we prove Theorem \ref{theorem:sn_at_most_tw_times_delta} and show an upper bound on the scramble number of planar graphs. %Section \ref{appendix} outlines an unpublished result by David Eppstein on the number of connected subgraphs in minor closed families of graphs. \\ 
\\

\noindent \textbf{Acknowledgements.} The authors were supported by Williams College and by the NSF via grants DMS-2241623 and DMS-1947438.  They thank Marchelle Beougher, Nila Cibu, Cassie Ding, and Chan Lee for many helpful conversations on scrambles and gonality.  % We also thank David Eppstein for allowing us to publish his result on minor closed families of graphs for the first time.  

\section{Preliminaries} \label{preliminaries}

In this section, we cover necessary background on terminology and notation for complexity theory and graph theory. We discuss scramble number and screewidth, and define a new invariant known as carton number, which is the main topic of this paper.  

\subsection{Complexity theory} We begin by recalling big-oh and big-omega notation.   For a function $f(n)$, we write $f(n)=O(g(n))$ if there exists a positive constant $C$ such that $f(n)\leq C \cdot g(n)$ for all sufficiently large $n$.  Similarly, we write $f(n)=\Omega(g(n))$ if there exists a positive constant $C$ such that $f(n)\geq C \cdot g(n)$ for all sufficiently large $n$.

We let $L \subset \{0, 1\}^*$ be a set of inputs. The corresponding decision problem is the computational task associated with determining whether a given input $x \in \{0, 1\}^*$ lies within $L$ or not. A decision problem is said to be in NP if a ``yes'' instance can be solved in nondeterministic polynomial time, or equivalently if a certificate for a solution can be checked in deterministic polynomial time.  Similarly a decision problem is said to be in co-NP if a ``no'' instance can be solved in nondeterministic polynomial time. A (not necessarily decision) problem is called NP-hard if there exists a polynomial time reduction from every problem in NP to it. In particular, there exists a polynomial time reduction from problem A to B if: given oracle access to a polynomial-time algorithm for B, there exists a polynomial time algorithm for A. 

A problem is fixed parameter tractable with respect to a parameter $k$ if, whenever $k$ is fixed, there exists a polynomial time solution.  A problem is in APX if it allows for a polynomial-time approximation algorithm with an approximation ratio bounded by a constant. We call a problem APX-hard if every problem in APX can be reduced to it in polynomial time while preserving the approximation.  If a (non-randomized) algorithm can compute a solution within a fixed ratio $\rho$, we call it a deterministic $\rho$-approximation.

\subsection{Graphs} 

Throughout this paper, a graph $G = (V, E)$ is a finite, connected, loopless multigraph with vertex set $V=V(G)$ and undirected edge multiset $E=E(G)$. This means that in the edge multiset $E$, we allow multiple edges between any two distinct vertices of $V$ but disallow edges from a vertex to itself. If there is at most one edge between any two vertices, a graph $G$ is said to be \textit{simple}. The \textit{girth} of a graph is the length of its smallest cycle. 

The \textit{degree} of a vertex $v$, denoted $d(v)$, counts the number of edges incident to $v$. For a subset $V' \subseteq V$, the subgraph $G[V']$ \textit{induced} by $V'$ has vertex set $V'$ and edge set $E' = \{(u,v) \in E | u, v \in V' \} \subseteq E$. The \textit{vertex connectivity} of a graph $G$, denoted $\kappa(G)$, is the minimum cardinality of a set $T \subseteq V$ for which $G[V \backslash T]$ is disconnected. Similarly, the \textit{edge connectivity} of a graph $G$, denoted, $\lambda(G)$, is the minimum cardinality of a set $W \subseteq E$ for which the graph $G' = (V, E \backslash W)$ is disconnected. In general, the \textit{k-restricted edge connectivity} of a graph $G$, denoted $\lambda_k(G)$, is the minimum cardinality of a set $W \subseteq E$ for which the graph $G' = (V, E \backslash W)$ is disconnected and every component of $G'$ has size $\geq k$ (if such a set $W$ exists; otherwise $\lambda_k(G)$ is undefined). We let $\xi_k(G)$ denote the minimum outdegree of any $k$-vertex-connected subgraph of $G$. If $\lambda_k = \xi_k(G)$, we say that $G$ is $\lambda_k$\textit{-optimal}.  

The \textit{independence number} of a graph $G$, denoted $\alpha(G)$, is the maximum cardinality subset $S \subseteq V$ for which no two vertices in $S$ share an edge. In general, the \textit{k-component independence number} of a graph $G$, denoted $\alpha_k(G)$, is the maximum cardinality of a subset $S \subseteq V$ for which every component of $G[S]$ has size $\leq k$. 

We now recall several operations and transformations on graphs. The \textit{Cartesian product} of graphs $G = (V_G, E_G)$ and $H = (V_H, E_H)$, denoted by $G \square H$, has vertex set $V_G \times V_H$ with edges between vertices $(u, u')$ and $(v,v')$ if and only if $u = v$ and $(u', v') \in E_H$ or $u' = v'$ and $(u, v) \in E_G$; if $G$ and $H$ are not both simple, edges in the multiset $E_{G\square H}$ appear as many times as the corresponding edge in $E_G$ or $E_H$. 

A degree $2$ vertex $v$ with distinct neighbors $s$ and $w$ in $G$ can be \textit{smoothed} by deleting $v$, deleting the edges $(u,v)$ and $(v,w)$, and adding the edge $(u,w)$. If $H$ can be obtained via a series of smoothings of $G$, then $G$ is known as a \textit{refinement} of $H$. An \textit{edge subdivision} is the opposite of smoothing: given vertices $u,w$ and the edge $(u,w)$, this operation deletes $(u,w)$, creates a new vertex $v$, and creates the edges $(u,v)$ and $(v,w)$. A graph $H$ obtained from a series of edge subdivisions of $G$ is a \textit{subdivision} of $G$. 

Distinct vertices $u, v, w$ with edges $(u,v)$ and $(v,w)$ can be \textit{lifted} by deleting the edges $(u,v)$ and $(v,w)$ while creating the edge $(u,w)$. A graph $H$ obtained via a series of \textit{lifts} of $G$ is known as an \textit{immersion minor} of $G$. A function $f$ is \textit{immersion minor monotone} if for every $H$ that is an immersion minor of $G$, we have $f(H)\leq f(G)$. 

An edge contraction is an operation in which two vertices connected by an edge are merged into one vertex; formally,given a graph $G$ with $u,v\in V$ and $(u,v)\in E$, an edge contraction is a map that is the identity on $V\setminus\{u,v\}$ and maps $u$ and $v$ to a new vertex $w$; for all $(x,z)\in E$ with $x,z\notin\{u,v\}$, it is again the identity, while the edges $(x,u)$ or $(x,v)$ are mapped to $(x,w)$. A graph $H$ obtained via a series of edge deletions, vertex deletions, and edge contractions of $G$ is known as a \textit{minor} of $G$. A function $f$ is said to be \textit{minor monotone} if $f(H) \leq f(G)$ when $H$ is a minor of $G$. A graph family $\mathcal{G}$ is said to be \textit{minor closed} if for every $G \in \mathcal{G}$, every minor of $G$ is also in $\mathcal{G}$. 

\subsection{Treewidth and Scramble Number}

The treewidth of a graph $G$, denoted $\tw(G)$, is a well-studied invariant that serves as a lower bound on gonality \cite{debruyn2014treewidth}. We omit the usual definition of treewidth in terms of tree-decompositions, and instead present an equivalent characterization with an invariant known as \textit{bramble number}.

A \emph{bramble} $\mathcal{B}$ on a graph $G$ is a collection of nonempty connected subgraphs of $G$ for which any two have non-empty intersection or are joined by an edge.  The \textit{order of a bramble } $\mathcal{B}$ is its \textit{hitting number} $h(\mathcal{B})$, the smallest cardinality of a vertex subset $H\subset V(G)$ such that $V(B) \cap H \neq \emptyset$ for all $B \in \mathcal{B}$. The \textit{bramble number} of a graph $G$, denoted $\bn(G)$ is then the maximum order of a bramble on $G$. In \cite{st-graph-searching-treewidth}, Seymour and Thomas proved that $\tw(G) = \bn(G) - 1$.

To find a more powerful lower bound on graph gonality, \cite{new_lower_bound} generalized brambles to scrambles.  A \emph{scramble} $\mathcal S$ on a graph $G$ is a nonempty collection of nonempty connected subgraphs of $G$, which we call the \emph{eggs }of the scramble.

The order of a scramble is defined in terms of its hitting number $h(\mathcal{S})$ and another quantity known as its egg-cut number $e(\mathcal{S})$. An \emph{egg-cut} of $\mathcal{S}$ is a collection of edges in $E(G)$ that, when deleted, disconnect $G$ into two components, each containing an egg from $\mathcal{S}$.  The \textit{egg-cut number} of $\scramble$, denoted $e(\scramble)$, is the minimum size of an egg-cut of $\mathcal{S}$; if no egg-cut exists (that is, if each pair of eggs of $\mathcal{S}$ overlap in at least one vertex), we set $e(\mathcal{S})=\infty$. The \textit{order of a scramble} $\scramble$, denoted $||\mathcal{S}||$ is then defined as the minimum of $h(\scramble)$ and $e(\scramble)$. The \textit{size} of a scramble, denoted $|\mathcal{S}|$, is the number of eggs in the scramble $\mathcal{S}$. 

Finally, the \textit{scramble number} of a graph $G$, written $\sn(G)$, is the largest possible order of a scramble on $G$:
\[\sn(G)=\max_{\stackrel{\mathcal{S}\textrm{ a scramble}}{\textrm{ on $G$}}}||\mathcal{S}||.\]
The following result shows that scramble number serves as a better lower bound on gonality than treewidth.

\begin{theorem} \cite[Theorem 1.1]{new_lower_bound} For any graph $G$,
\[\tw(G)\leq \sn(G)\leq \gon(G).\]
    
\end{theorem}

 A \textit{disjoint scramble} $\scramble$ on a graph $G$ is a scramble on $G$ such that the eggs of $\scramble$ are pairwise disjoint.  Note that the hitting number of a disjoint scramble is simply its size.
The \textit{disjoint scramble number} of a graph $G$, written $\dsn(G)$, is the largest order of a disjoint scramble on $G$. Note that $\dsn(G)\leq \sn(G)$; sometimes this inequality is strict \cite[Proposition 5.2]{sn_two}.

\begin{figure}[hbt]  
    \centering
    \includegraphics[scale = 0.45]{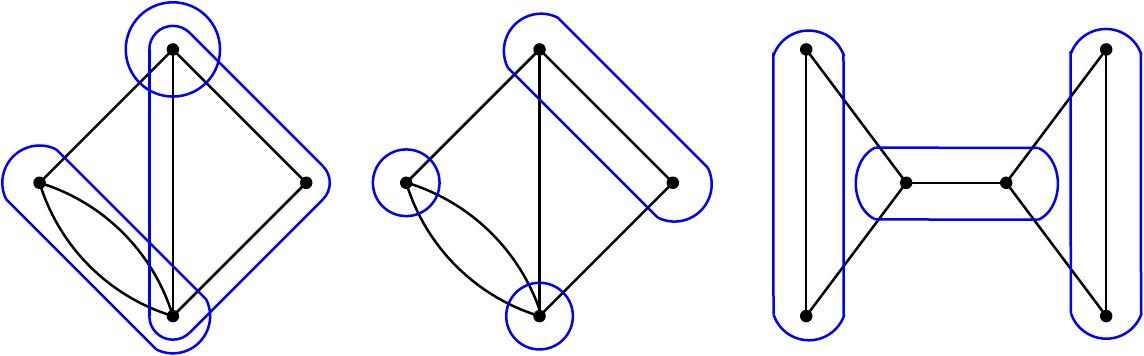}
    \caption{Scrambles of orders $2$, $3$, and $1$, respectively.  The second and third scrambles are disjoint, while the first is not.}
    \label{figure:three_scrambles}
\end{figure}

Figure \ref{figure:three_scrambles} illustrates several scrambles, the first two on the same graph $G_1$.  The first scramble $\mathcal{S}_1$ on $G_1$ has $h(\mathcal{S}_1)=2$ and $e(\mathcal{S}_1)=3$, so that $||\mathcal{S}_1||=2$.  The second scramble $\mathcal{S}_2$ on $G_1$ has $h(\mathcal{S}_2)=3$ and $e(\mathcal{S}_2)=3$, so that $||\mathcal{S}_2||=3$. This turns out to be the maximum order of a scramble on $G$: to have a larger hitting number, each vertex must be its own egg, but then there would exist an egg-cut of size $2$.  Thus we have $\sn(G_1)=3$, and since this can be achieved by the disjoint scramble $\mathcal{S}_2$ we have $\dsn(G_1)=3$.  The third scramble $\mathcal{S}_3$ on the final graph $G_2$ has order $1$, since the middle edge forms an egg-cut; this illustrates that there may be smaller egg-cuts than the minimum number of edges leaving any egg.

Given a scramble $\scramble$ on some graph $G$, a \textit{subset scramble} of $\scramble$ is another scramble $\scramble'$ on $G$ such that $\scramble' \subseteq \scramble$, i.e. each egg in $\scramble'$ is also an egg in $\scramble$.  Given a graph $G$ and a positive integer $k$, the \emph{$k$-uniform scramble}, denoted $\varepsilon_k$, on $G$ is the scramble whose eggs are all the connected subgraphs of $G$ on $k$ vertices \cite{uniform_scrambles}. In particular, the $1$-uniform scramble is the scramble where the eggs are precisely the vertices of $G$; these are sometimes called verteggs. The 2-uniform scramble is the scramble where the eggs are all pairs of vertices that are connected by an edge; it is sometimes known as the edge scramble. 

\subsection{Carton Number}

%The scramble number of many families of graphs has been studied before \cite{uniform_scrambles,echavarria2021scramble,rooks_gonality}.
 While scramble number is NP-Hard to compute \cite[Theorem 1.1]{echavarria2021scramble}, its inclusion in NP is unknown. Scrambles seem to be a natural NP-certificate candidate for lower-bounding scramble number, but the size of maximal order scrambles may be exponentially large in the size of the graph. Bounding the size of maximal order scrambles makes brute force computation of scramble number, where every possible combination of scrambles is considered, more efficient. In particular, a polynomial upper bound on the size of maximum order scrambles would give for a polynomial-size certificate for scramble number (though not necessarily immediately implying checkability in polynomial time). With this in mind, we define the \textit{carton number} of a graph $G$, denoted $\cart(G)$, as the smallest number of eggs in any scramble $\scramble$ with order equal to $\sn(G)$. In other words, it is the minimal size of a maximal order scramble; or formulaically,
\[\cart(G)=\min\{|\mathcal{S}|\,:\, ||\mathcal{S}||=\sn(G)\}.\]A \textit{carton scramble} on a graph $G$ is a scramble with order $\sn(G)$ and $\cart(G)$ many eggs.

\begin{figure}[hbt]
    \centering
    \includegraphics[scale = 0.5]{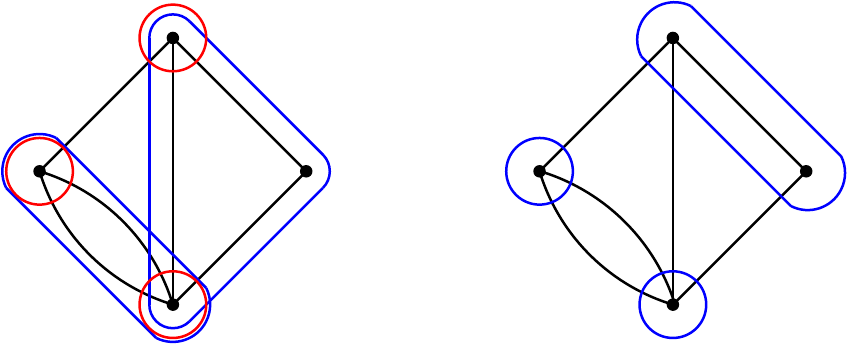}
    \caption{Two scrambles of order three on a graph with carton number three.}
    \label{figure:carton_number}
\end{figure}

For instance, Figure \ref{figure:carton_number} presents two scrambles, both of order $3$, on a graph $G$ of scramble number $3$.  The first scramble has size $5$, while the second has size $3$.  No scramble of maximal order on the graph has size smaller than $3$, since $h(\mathcal{S})\leq |\mathcal{S}|$.  Thus $\textrm{cart}(G)=3$, and the second scramble is a carton scramble while the first is not.

\subsection{Screewidth}
While lower-bounding scramble number is possible with a specific scramble, upper-bounding is more subtle, especially if there is a gap between scramble number and gonality. 
One useful invariant for providing upper bounds is \emph{screewidth}, recently developed in \cite{screewidth-og}. There are no known cases where screewidth is larger than gonality, meaning it is in practice a more useful upper bound.

For a graph $G$, a \textit{tree-cut decomposition of $G$} is a pair $\mathcal{T} = (T, \mathcal{X})$ where $T$ is a tree\footnote{A tree is a connected acyclic graph.} and $\mathcal{X}=(X_b)_{b\in V(T)}$ is a family of subsets of $V(G)$, such that each vertex $v\in V(G)$ appears in exactly one $X_b$.
Thus $\mathcal{X}$ forms a near-partition of $V(G)$, which is the same as a partition with empty sets allowed.  For clarity, we refer to the vertices and edges of $T$ as \emph{nodes} and \emph{links}, respectively, reserving ther terms \emph{vertices} and \emph{edges} for $G$.  We refer to the set $X_b$ as the \emph{bag} associated to the node $b\in V(T)$.

The \textit{adhesion} of a link $l \in E(T)$, denoted as $\adh(l)$, is the set of edges $(v_a, v_b) \in E(G)$ where $v_a$ and $v_b$ are in different components of $T - l$. Similarly, the adhesion of a non-leaf\footnote{The adhesion of a leaf node, with degree equal to $1$, is taken to be the empty set.} node $b \in V(T)$ is the set of edges $(v_a, v_b) \in E(G)$ where $v_a$ and $v_b$ are in different components of $T - b$. The \textit{width} of the tree-cut decomposition $\mathcal{T}$ is then defined to be $w(\mathcal{T}) = \max(\lw(\mathcal{T}), \bw(\mathcal{T}))$ where 
\begin{align*}
    \lw(\mathcal{T}) = \max_{l \in E(T)} |\adh(l)|, \quad \quad \bw(\mathcal{T}) = \max_{b \in V(T)} \left[ |X_b| + |\adh(l)| \right].
\end{align*}
Finally, the \emph{screewidth} of $G$ is the minimum width of any tree-cut dcomposition of $G$; that is, $\scw(G) = \min_{\mathcal{T}} w(\mathcal{T})$.  By \cite[Theorem 1.1]{screewidth-og}, we have $\sn(G)\leq \scw(G)$ for any graph $G$.

We can visualize a tree-cut decomposition by drawing a thickened copy of $T$, drawing the vertices of $G$ inside of their corresponding nodes, and drawing the edges $(u,v)\in E(G)$ passing through the unique path in $T$ connecting the vertices $u$ and $v$ in their respective bags.  The adhesion of a link is then the set of edges passing through it; and the adhesion of a non-leaf node is the set of edges ``tunneling'' through the node, without either endpoint in that node.  Thus $\lw(\mathcal{T})$ is the maximum number of edges passing through a link; and $\bw(\mathcal{T})$ is the maximum over all nodes of the number of vertices plus the number of tunneling edges.

Two examples of tree-cut decompositions $\mathcal{T}_1$ and $\mathcal{T}_2$ of the same graph $G$ are illustrated in Figure \ref{figure:screewidth_example}.  The first has $\lw(\mathcal{T}_1)=5$ and $\bw(\mathcal{T}_1)=6$, so $w(\mathcal{T}_1)=6$.  The second has $\lw(\mathcal{T}_2)=3$ and $\bw(\mathcal{T}_2)=2$, so $w(\mathcal{T}_2)=3$.  It follows that $\scw(\mathcal{T})\leq 3$.  As previously established, this graph has scramble number $3$, implying that $\scw(G)=3$ since $\sn(G)\leq \scw(G)$.

\begin{figure}[hbt]
    \centering
    \includegraphics[scale = 0.7]{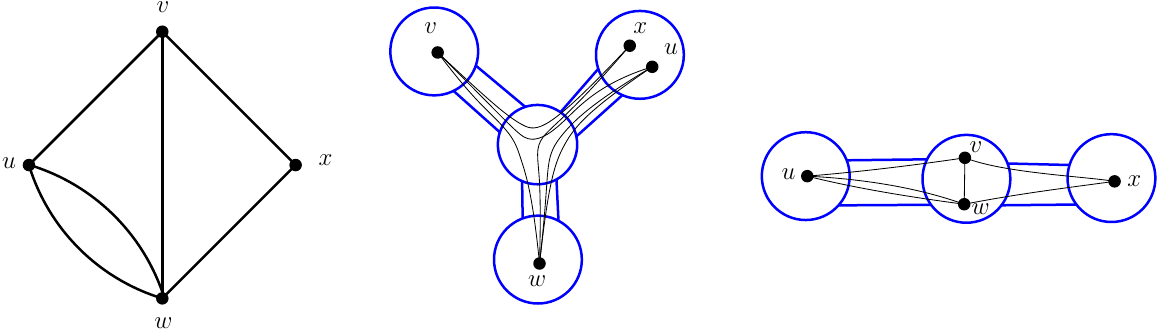}
    \caption{A graph with two tree-cut decompositions of it, one of width $6$ and one of width $3$}
    \label{figure:screewidth_example}
\end{figure}

\subsection{Premilinary results}

Below, we prove or recall results that will be useful throughout the paper. Some follow quickly from definitions and others come from previous work. 

\begin{proposition} \label{dsn-sn-cart-inequality-chain} For any graph $G$, we have
    $\dsn(G) \leq \sn(G) \leq \cart(G)$.
\end{proposition}
\begin{proof}
    The first inequality holds since any disjoint scramble is a scramble; $\sn(G)$ is the largest number in the set of possible scramble orders, while $\dsn(G)$ is the largest element of a subset of that set. For the second inequality, fix any maximal order scramble $\mathcal S$. Then $\sn(G) = \min\{h(\scramble), e(\scramble)\} \leq h(\mathcal S)$ and $h(\mathcal S)$ is no greater than the number of eggs in $\mathcal S$.
    \end{proof}

Since $\dsn(G)\leq \sn(G)$, we have that disjoint scramble number also serves as a lower bound on graph gonality, although it will never perform better than scramble number.  We will see in Section \ref{properties-disjoint-scramble-number} that disjoint scramble number may be larger than, smaller than, or equal to treewidth, meaning it also does not consistently perform better than treewidth as a lower bound on gonality.  However, disjoint scramble number has the advantage that the problem of deciding whether $\dsn(G)\geq k$ is in NP, with a disjoint scramble of order $k$ serving as a polynomial size certificate verifiable in polynomial time (for a disjoint scramble $\mathcal{S}$ we have $h(\mathcal{S})=|\mathcal{S}|$, and $e(\mathcal{S})$ can be computed in polynomial time as there are polynomially many eggs \cite{uniform_scrambles}).  Lower bounding treewidth or scramble number is not known to be in NP; indeed, since upper bounding treewidth is NP-complete, lower bounding could only be in NP if NP$=$co-NP, which is generally expected not to be the case.

\begin{proposition} \label{egg-cut-subset-scramble}
    Given any graph $G$ with scramble $\mathcal S$, for any subset scramble $\mathcal S' \subseteq \mathcal S$ we have $e(\mathcal S') \geq e(\mathcal S)$.  
\end{proposition}

\begin{proof}
    Since $\mathcal{S}'\subset \mathcal{S}$, any egg-cut for $\mathcal{S}'$ is also an egg-cut for $\mathcal{S}$.  Since there exists a set of $e(\scramble')$ edges that forms an egg-cut for $\mathcal{S}'$, the same is true for $\mathcal{S}$, implying that $e(\mathcal{S})\leq e(\mathcal{S}')$.
\end{proof}
This inequality can be strict in some cases, as $\scramble$ might contain pairs of eggs that $\scramble'$ does not contain which can be separated by removing fewer than $e(\scramble')$ edges. An extreme example is where $\scramble'$ contains only one egg, so $e(\scramble') = \infty$, even for $\scramble$ with $e(\scramble)$ finite.

\begin{proposition} \label{scramble-order-of-brambles}
     \cite[Lemma 3.6]{new_lower_bound} Let $\mathcal{B}$ be a bramble of bramble order $k$. Then, $\mathcal{B}$ is a scramble with scramble order either $k$ or $k - 1$.
\end{proposition}

\begin{proposition} \label{prop:sn_invariant_subdivision}  
\cite[Proposition 4.6]{new_lower_bound} Scramble number is invariant under subdivision and smoothing.
\end{proposition}

\begin{proposition}\label{prop:specific_scrambles_subdivision}  Suppose that $G'$ is a subdivision of $G$ obtained by adding a single vertex $v$ in between two adjacent vertices $u$ and $w$. 
\begin{itemize}
    \item[(i)] Let $\mathcal{S}$ be a scramble on $G$, and construct $\mathcal{S}'$ from $\mathcal{S}$ by adding $v$ to every egg of $\mathcal{S}$ that contains $u$.  Then $\mathcal{S}'$ is a scramble on $G'$ with $||\mathcal{S}'||\geq ||\mathcal{S}||$.
    \item[(ii)] Let $\mathcal{S}'$ be a scramble on $G'$ of order at least $3$, and construct $\mathcal{S}$ from $\mathcal{S}'$ by deleting $v$ from every egg that contains it.  Then $\mathcal{S}$ is a scramble on $G$ with $||\mathcal{S}||\geq ||\mathcal{S}'||$.
\end{itemize}
\end{proposition}

\begin{proof}
    Both claims are established as part of the proof of  \cite[Proposition 4.6]{new_lower_bound}.
\end{proof}

\begin{proposition} \label{dsn=sn}
 For all graphs with $\sn(G) \leq 3$, we have that $\sn(G) = \dsn(G)$.  Moreover, if $\dsn(G)\leq 2$, then $\sn(G)=\dsn(G)$.
\end{proposition}

\begin{proof}
The first claim is the content of \cite[Proposition 5.2]{sn_two}.  To prove the second claim, then by the first claim it's enough to show that if $\sn(G)\geq 3$ then $\dsn(G)\geq 3$.  This follows from the proof of \cite[Proposition 5.2]{sn_two}, where it is argued that a graph with a scramble of order $3$ also has a disjoint scramble of order $3$.
\end{proof}

\section{Properties of Carton Number} \label{carton-props-section}

In this section we establish properties of carton number.  Before we do so, we prove a useful statement regarding the relation between the order of a scramble and the hitting number of a corresponding equal-order subset scramble. 

\begin{proposition} \label{subset-scrambles-hitting}
    Given any graph $G$ with scramble $\mathcal S$, there exists a subset scramble $\mathcal S' \subseteq \mathcal S$ such that $||\mathcal S'|| = ||\mathcal S||$ and $h(\mathcal S') = ||\mathcal{S}'||$. 
\end{proposition}

In other words, any scramble has a subset scramble of equal order, and that order is given by the subset scramble's hitting number.

\begin{proof}
   If $h(\mathcal S) = ||\mathcal S||$, then we are immediately done.  Otherwise, it must be the case that $h(\mathcal S) > ||\mathcal S||$ by definition of the order of a scramble.  We claim that there exists some proper subset $\mathcal S'$ of $\mathcal S$ such that $||\mathcal{S}|| = ||\mathcal S'|| = h(\mathcal S')$.
    
    To see this is the case, denote $\mathcal S_0 = \mathcal S$, and starting from $i = 0$ iteratively obtain $\mathcal S_{i+1}$ from $\mathcal S_i$ by arbitrarily removing one egg, doing so for all $0 \leq i \leq m - 1$ where $m = |\mathcal S|$.  Then $h(\mathcal S_i) \geq h(\mathcal S_{i+1})$ as any hitting set for $\mathcal{S}_i$ is also a hitting set for $\mathcal S_{i+1}$.  However, we also have that $h(\mathcal S_i) \leq h(\mathcal S_{i+1})+1$ since any hitting set of $\mathcal S_{i+1}$, together with an arbitrary vertex of the egg deleted from $\mathcal{S}_i$, is a hitting set for $\mathcal{S}_i$.  Thus, for each integer $0 \leq i \leq m-1$, we have that
    \[ h(\mathcal S_i) - 1 \leq h(\mathcal S_{i+1}) \leq h(S_i).\]
    As $h(\mathcal S_0) > ||\mathcal S||$ and $h(\mathcal S_{m-1}) = 1 < ||\mathcal S||$, the above inequality and the fact that all values are integers implies that there exists $j$ such that $h(\mathcal S_j) = ||\mathcal S||$. Furthermore, by Proposition \ref{egg-cut-subset-scramble}, necessarily $e(\mathcal S) \leq e(\mathcal S_j)$.  As $||\mathcal S|| \leq e(\mathcal S) \leq e(\mathcal S_j)$, it follows that
    \[||\mathcal S_j|| = \min(h(\mathcal S_j), e(\mathcal S_j)) = \min(||\mathcal S||, e(\mathcal S_j)) = ||\mathcal S||.\]
    Thus, letting $\mathcal S' = \mathcal S_j$, we have that $\mathcal S'$ is a subset scramble of $\mathcal S$ of equal order such that $h(\mathcal S') = ||\mathcal S'||$ as desired. 
\end{proof}

From this, we immediately obtain the following result about the existence of maximum order scrambles with hitting number equal to scramble number.

\begin{corollary} \label{max-order-scramble-hitting}
    Given any graph $G$ with maximum order scramble $\mathcal S$, there exists a subset scramble $\mathcal S'\subset \mathcal S$, also of maximum order, such that $h(\mathcal S') = \sn(G)$.
\end{corollary}

\begin{proof}
    Given a maximum order scramble $\mathcal S$ of $G$, apply Proposition \ref{subset-scrambles-hitting} to find $\mathcal S'\subset \mathcal S$ such that $||\mathcal S'|| = || \mathcal S||$ and $h(\mathcal S') = ||\mathcal S'||$.  As $\mathcal S$ is a maximum order scramble and $\mathcal S'$ has the same order, we have $||\mathcal S'|| = \sn(G)$.
\end{proof}

Finally, we relate this result to carton number.

\begin{corollary}
  For any carton scramble $\mathcal S_{\cart}$ on a graph $G$, we have $h(\mathcal S_{\cart}) = \sn(G)$.
\end{corollary}

\begin{proof}
    By Corollary \ref{max-order-scramble-hitting} and the fact that $\mathcal S_{\cart}$ must be a maximum order scramble of $G$, necessarily $\mathcal S_{\cart}$ must contain a maximum order scramble $\mathcal S'$ such that $h(\mathcal S') = \sn(G)$.  However, if $S'$ was a proper subset of $\mathcal S_{\cart}$, then this would contradict the fact that $\mathcal S_{\cart}$ is a carton scramble, as $\mathcal S'$ would be a maximum order scramble of size less than $\mathcal S_{\cart}$.  Thus, it follows that $\mathcal S' = \mathcal S_{\cart}$, and the result follows as desired.  
\end{proof}

%This style of proof, namely paring down arguments for reducing the size of scrambles, naturally leads to the question of whether or not all maximum order scrambles can be pared down to a carton scramble.  If this were the case and we could also show that all scrambles larger than some polynomial size could be pared down without altering their order, then it could immediately be shown that all graphs have a polynomial size carton number.  This, however, turns out not to be the case, as will be seen in the next example.

Given the proof strategy of paring down scrambles, it is natural to ask whether all maximum order scrambles can be pared down to a carton scramble; or perhaps to a scramble of polynomial size.  The following result shows that this is not the case.

\begin{proposition}
    Let $G$ be a graph on $n$ vertices, and $\mathcal{S}$ a maximum order scramble on $G$.  Then there exists a subset scramble $\mathcal{S}'\subseteq \mathcal{S}$ such that $|\mathcal{S}'|\leq {n \choose \lfloor n/2\rfloor}$.  Moreover, this bound is sharp for $n\geq 6$:  for every such $n$, there exists a graph $G$ on $n$ vertices and a maximum order scramble $\mathcal{S}$ with $|\mathcal{S}|={n \choose \lfloor n/2\rfloor}$ such that deleting any egg from $\mathcal{S}$ decreases the order of the scramble, and such that $\mathcal{S}$ is not a carton scramble.
\end{proposition}

\begin{proof}
    Given a maximal order scramble $\mathcal{S}$ on $G$, define $\mathcal{S}'$ by iteratively deleting eggs from $\mathcal{S}$ that contain any other eggs of $\mathcal{S}$.  As with any subset scramble, we have $h(\mathcal{S}')\leq h(\mathcal{S})$ and $e(\mathcal{S}')\geq e(\mathcal{S})$.  Moreover, since any egg of $\mathcal{S}$ contains an egg of $\mathcal{S}'$ as a subgraph, any hitting set for $\mathcal{S}'$ is a hitting set for $\mathcal{S}$; thus $h(\mathcal{S}')= h(\mathcal{S})$.  And, any egg-cut for $\mathcal{S}$, say separating eggs $E_1$ and $E_2$, is also an egg-cut for $\mathcal{S}'$, as it separates eggs that are subgraphs of $E_1$ and of $E_2$. Thus $e(\mathcal{S}')=e(\mathcal{S})$. It follows that $||\mathcal{S}'||=||\mathcal{S}||=\sn(G)$.

    To bound the size of $\mathcal{S}'$, we note that the vertex sets of its eggs form a collection of distinct, nonempty subsets of $V(G)$, none of which contain each other.  The largest number of subsets of an $n$-element set that are pairwise incomparable is ${n \choose \lfloor n/2\rfloor}$, achieved for instance by choosing all size $\lfloor n/2\rfloor$ -- this is known as Sperner's Theorem.  Thus this forms a bound on $|\mathcal{S}'|$.

    To show this bound is sharp, assume first that $n$ is even, and write $n=2m$.  Construct $G$ first by considering $K_{m+1,m-1}$, the complete bipartite graph on $m+1$ and $m-1$ vertices; and then add edges to the set of $m+1$ vertices so that they form a cycle. Figure \ref{figure:K_46_cycle} illustrates this graph for $n=10$.  Since $G$ is simple and the minimum degree is at least $\lfloor n/2\rfloor +1$, we may apply \cite[Corollary 3.2]{echavarria2021scramble} to deduce that $\sn(G)=n-\alpha(G)=2m-(m-1)=m+1$.

    \begin{figure}[hbt]
    \centering
    \includegraphics[scale = 1]{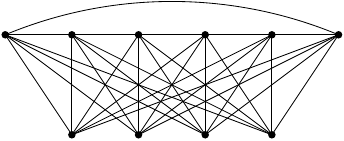}
    \caption{A complete bipartite graph $K_{6, 4}$ with cycle added on the side with $6$ vertices.}
    \label{figure:K_46_cycle}
\end{figure}

     We claim that $\kappa(G)=m+1$.  Indeed, as long as at least one vertex from each of the two sides remains, the graph is connected.  So either the $m+1$ vertices on the larger side must be deleted; or the $m-1$ vertices on the smaller side must be deleted, followed by enough vertices on the larger side to disconnect the cycle, adding two.  Thus the $m$-uniform scramble $\varepsilon_m$ on $G$, whose eggs consist of all connected subgraphs on $m$ vertices, has $\binom{n}{m}$ eggs.  Moreover, $h(\varepsilon_m)=m+1$, since any set $S$ of $m$ vertices fails to hit at least one egg (namely the egg whose vertices are $V(G)-S$); and $e(\varepsilon_m)\geq \lambda(G)\geq \kappa(G)=m+1$.  Thus $\varepsilon_m$ is a maximal order scramble.  However, deleting a single egg would decrease its order, since the set of vertices not in that egg would form a hitting set of size $m$.  Thus we have a maximal order scramble of size $n\choose \lfloor n/2\rfloor$ that cannot be pared down further.  It is not, however, a carton scramble:  by \cite[Lemma 3.1]{echavarria2021scramble}, the $2$-uniform scramble $\varepsilon_2$ also has order $n-\alpha(G)$, and with ${n\choose 2}$ eggs it is smaller than $\varepsilon_m$ since $n\geq 6$.

     For $n$ odd, say $n=2m+1$, construct $G$ as $K_{m+1,m}$ and add edges to the $m+1$ vertices to form a cycle.  The same argument as above shows $\sn(G)=n-\alpha(G)=2m+1-m=m+1$; that $\kappa(G)=m+1$;  that $|\varepsilon_{m+1}|={n \choose m+1}$; that $||\varepsilon_{m+1}||=m+1$; that removing any eggs from $\varepsilon_{m+1}$ decreases its order; and that $\varepsilon_2$ has strictly fewer eggs while still achieving scramble number.
\end{proof}
%It is worth noting that the above argument does not yet imply that carton number can be exponential in $|V(G)|$.   Although the example maximal order scrambles could not be pared down below $n\choose \lfloor n/2\rfloor$, the example graphs do have maximal order scrambles that are polynomial in size; for instance, by \cite[]{echavarria2021scramble} the $2$-uniform scramble achieves scramble number with $O(n^2)$ eggs.
    
We now characterize precisely when carton number and scramble number are equal.

\begin{theorem} \label{cart=sn}
    For any graph $G$, $\cart(G) = \sn(G)$ if and only if $\dsn(G) = \sn(G)$.
\end{theorem}

\begin{proof}
    Let $G$ be any graph. We will prove the implication in each direction.
        
    First suppose $\dsn(G) = \sn(G)$. This means there exists a disjoint scramble $\scramble$ with $\| \scramble \| = \sn(G)$.  By Proposition \ref{max-order-scramble-hitting}, there exists some maximum order scramble $\scramble'$ such that $\scramble' \subseteq \scramble$ and $h(\scramble') = \sn(G)$. As a subset of a disjoint scramble, $\scramble'$ is itself a disjoint scramble, implying $h(\scramble') = |\scramble ' |$, and so $|\scramble ' | = \sn(G)$.  As $\mathcal S'$ is a maximum order scramble we have $\cart(G) \leq |\scramble ' | = \sn(G)$, and as $\cart(G) \geq \sn(G)$ for any graph we conclude that $\cart(G) =\sn(G)$.

    Now suppose $\cart(G) = \sn(G)$. Then there exists a scramble $\scramble$ with $\|\scramble\| = |\scramble| = \sn(G)$. If $\scramble$ is not disjoint, then some pair of eggs share a vertex. But then there is a hitting set that includes that vertex and one vertex from each other egg. This hitting set has size $\sn(G) - 1$, which is a contradiction to $\|\scramble\| = \sn(G)$. Therefore $\scramble$ is disjoint, so $\dsn(G) = \sn(G)$.
\end{proof}

%This immediately lets us compute carton number for complete graphs $K_n$, complete bipartite graphs $K_{m,n}$, and more generally complete multipartite graphs $K_{n_1,\ldots,n_\ell}$, which have $n=\sum_{i}n_i$ vertices clustered in $\ell$ sets of size $n_1,\ldots,n_\ell$, with each vertex connected to every other vertex not in its set.   In particular, we have $\sn(K_{n_1,\ldots,n_\ell})=n-\max_i n_i$ (this follows from \cite[Example 4.3]{debruyn2014treewidth} and the fact that $\tw(G)\leq \sn(G)\leq gon(G)$), and this can be achieved with a disjoint scramble consisting of single vertex subgraphs outside of the largest partite set.  Thus we have $\sn(K_{n_1,\ldots,n_\ell})=\dsn(K_{n_1,\ldots,n_\ell})=\cart(K_{n_1,\ldots,n_\ell})=n-\max_i n_i$

We can also combine our result with Proposition \ref{dsn=sn} to characterize carton number for graphs of scramble number at most $3$.

\begin{corollary} \label{cart=sn=dsn}
    For all graphs $G$ with $\sn(G) \leq 3$, we have that $\cart(G) = \sn(G) = \dsn(G)$.
\end{corollary}

\begin{proof}
Note that by Proposition \ref{dsn=sn}, any graph $G$ with $\sn(G) \leq 3$ satisfies $\sn(G) = \dsn(G)$, and by Theorem \ref{cart=sn} this then implies that $\cart(G) = \sn(G)$ as well.  
\end{proof} 

Another family of graphs we consider are those with an edge connectivity of 1, i.e. graphs which contain a bridge\footnote{In a connected graph, a bridge is an edge that, if deleted, disconnects the graph.}.  Given such a graph, it turns out to be trivial to compute its carton number given the carton and scramble numbers of its two constituent subgraphs connected by a bridge. 

\begin{proposition}
    Let $G$ be a graph with bridge $e = (u,v)$ such that removing $e$ from $G$ disconnects $G$ into two connected components $G_1$ and $G_2$. 
 Without loss of generality, we take  $\sn(G_1)\geq \sn(G_2)$; and if $\sn(G_1)=\sn(G_2)$, we take $\cart(G_1)\leq \cart(G_2)$.  Then, $\cart(G)=\cart(G_1)$.
    
\end{proposition}

\begin{proof} 
First we note that if $G$ is a tree, our claim holds as all scramble and carton numbers are equal to $1$.  We now assume $G$ is not a tree, so $\sn(G)\geq 2$.

By \cite[Lemma 2.4]{echavarria2021scramble}, we know \(\sn(G)=\max\{\sn(G_1),\sn(G_2)\}=\sn(G_1)\).  Any scramble on $G_1$, when considered as a scramble on $G$, has at least as high of an order:  hitting number remains unchanged, and egg-cut number could only go up.  Choose a carton scramble $\mathcal{S}'_\textrm{cart}$ on $G_1$, and let $\mathcal{S}$ be the same scramble on $G$.  Then,
\[\sn(G_1)=||\mathcal{S}'_\textrm{cart}||\leq ||\mathcal{S}||\leq \sn(G)=\sn(G_1).\]
It follows that all terms are equal and so $\mathcal{S}$ is a maximal order scramble on $G$. Since $\cart(G_1)=|\mathcal{S}'_\textrm{cart}|= |\mathcal{S}|$, we have $\cart(G)\leq \cart(G_1)$.

For the other direction, consider a carton scramble $\mathcal{S}_\textrm{cart}$ on $G$.  Since $\sn(G)\geq 2$, it cannot be that both $G_1$ and $G_2$ contain a complete egg; choose $i$ and $j$ so that $i\neq j$ and $G_j$ does not contain a complete egg.  As argued in the proof of \cite[Lemma 2.4]{echavarria2021scramble}, intersecting the eggs of $\mathcal{S}_\textrm{cart}$ with $G_i$ yields a scramble $\mathcal{S}'$ on $G_i$ with $||\mathcal{S}'||\geq ||\mathcal{S}_\textrm{cart}||$. Since
\[\sn(G)\geq \sn(G_i)\geq ||\mathcal{S}'||\geq||\mathcal{S}_\textrm{cart}||=\sn(G),\]
we know that all terms are equal.  In particular, $\sn(G)=\sn(G_i)=||\mathcal{S}'||$.

If $\sn(G_1)>\sn(G_2)$, then we have $i=1$ and $\mathcal{S}'$ is a maximum order scramble on $G_1$.  Since $|\mathcal{S}'|\leq |\mathcal{S}_\textrm{cart}|=\cart(G)$, we have $\cart(G_1)\leq \cart(G)$.  If $\sn(G_1)=\sn(G_2)$, then although we cannot immediately deduce the value of $i$, we do know that $\cart(G_i)\leq \cart(G)$.  Since $\cart(G_1)\leq \cart(G_i)$, we still conclude $\cart(G_1)\leq \cart(G)$.

\end{proof}

We now consider properties under which carton number is invariant, namely subdivision and smoothing.  

\begin{proposition} \label{subdivision-cart-equal}
    If $G'$ is a subdivision of $G$, then $\cart(G') = \cart(G)$
\end{proposition}

\begin{proof}
By induction, it will suffice to consider the case where $G'$ is obtained via the subdivision of a single edge between adjacent vertices $u$ and $w$, which introduces a new vertex $v$ between them.

Before considering arbitrary graphs, we note that this claim holds for any graph $G$ with $\sn(G)\leq 3$.  This is because for such graphs, we have
\[\cart(G)=\sn(G)=\sn(G')=\cart(G')\]
by Corollary \ref{cart=sn=dsn}.

For the remainder of the proof we can assume $\sn(G)\geq 4$. We first show that $\cart(G')\leq \cart(G)$.  Choose $\mathcal{S}$ to be a carton scramble on ${G}$, and let $\mathcal{S}'$ be the corresponding scramble on $G'$ from Proposition \ref{prop:specific_scrambles_subdivision}(i), which has order at least $||\mathcal{S}||=\sn(G)$.  Combined with Proposition \ref{prop:sn_invariant_subdivision}, this gives us \[\sn(G)=\sn(G')\geq ||\mathcal{S}'||\geq ||\mathcal{S}||=\sn(G),\]
so all terms are equal and $\mathcal{S}'$ is a maximum order scramble on $G'$.  The construction of $\mathcal{S}'$ left the number of eggs unchanged, so $\cart(G)=|\mathcal{S}|=|\mathcal{S}'|\geq \cart(G')$, as desired.

  Now we show $\cart(G)\leq \cart(G')$.  Choose a carton scramble $\mathcal{S}'$ on $G'$.  Since $||\mathcal{S}'||=\sn(G)\geq 4$, we can use Proposition \ref{prop:specific_scrambles_subdivision}(ii) to construct $\mathcal{S}$ on
 $G$ with $||\mathcal{S}||\geq ||\mathcal{S}'||$. This gives us
  \[\sn(G')=\sn(G)\geq ||\mathcal{S}||\geq ||\mathcal{S}'||=\sn(G'),\]
  so $||\mathcal{S}||$ is a maximum order scramble on $G$.  The construction of $\mathcal{S}$ could only decrease the number of eggs, so $\cart(G')=|\mathcal{S}'|\geq |\mathcal{S}|\geq \cart(G)$, as desired.

Having proven both inequalities, we conclude that $\cart(G') = \cart(G)$ as desired. 
\end{proof}

\begin{proposition}
    If $G'$ is a smoothing of $G$, then $\cart(G) = \cart(G')$
\end{proposition}
\begin{proof}
    Suppose $G'$ is a smoothing of $G$.  Then $G$ is a subdivision of $G'$, and hence $\cart(G) = \cart(G')$ by Proposition \ref{subdivision-cart-equal}.  
\end{proof}

We close this section by describing some operations under which carton number does not behave nicely; generally speaking, it can increase when passing to a ``simpler'' graph. First, it is not in general monotone (that is, non-increasing) under taking subgraphs. For instance, in the next section we will see that a $4\times 4$ rook's graph has carton number strictly greater than $16$.  It is a subgraph of the complete graph $K_{16}$, which has scramble number $15$ that can be obtained with the verteggs scramble (and thus has carton number at most 16).

Next, carton number is not monotone under edge contractions associated to taking minors, in which two vertices are merged; nor is it monotone under tunneling moves, in which edges $(u,v)$ and $(v,w)$ are deleted and an edge $(u,w)$ is added.  To see this, we refer to examples from previous literature:  \cite[Example 4.4]{new_lower_bound} and \cite[Example 2.8]{echavarria2021scramble}, reproduced in Figure \ref{figure:minor_and_immersion_minor}.  In the former, performing an edge contraction at $(g,h)$ turns a graph $G$ into a graph $G'$ with $\sn(G)=3$ and $\sn(G')=4$; we have $\cart(G')\geq \sn(G')=4$, and $\cart(G)=\sn(G)=3$ by Corollary \ref{cart=sn=dsn}, so carton number has increased.  In the latter, performing a pair of tunneling moves (first replacing $(a,c)$ and $(c,d)$ with $(a,d)$, then replacing $(a,d)$ and $(d,e)$ with $(a,e)$) turns a graph $H$ into a graph $H'$ with $\sn(H)=2$ and $\sn(H')=3$; we have $\cart(H)=\sn(H)=2$ and $\cart(H')=\sn(H')=3$ by Corollary \ref{cart=sn=dsn}, so carton number has increased. 

\begin{figure}[hbt]
    \centering
    \includegraphics[scale=0.8]{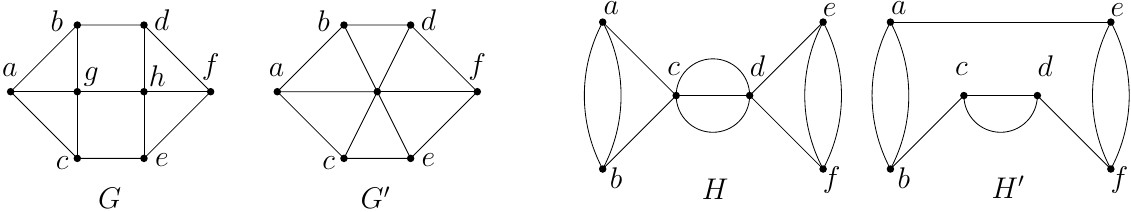}
    \caption{A graph $G$ with a minor $G'$ of higher carton number; and a graph $H$ with an immersion minor $H'$ of higher carton number.}
    \label{figure:minor_and_immersion_minor}
\end{figure}

%\begin{proposition}
 %   Carton number is not immersion minor monotone. 
%\end{proposition}

%To do this, we use the same graph from \cite[Example 2.8]{echavarria2021scramble} which was used to demonstrate that scramble number is not immersion minor monotone.

%\begin{figure}[H]
 %   \centering
%    \includegraphics[scale = 0.45]%{Figures/immersion_minor.pdf}
%    \caption{Graph $H$ with $\cart(H) = 2$ and its immersion minor $G$ with %$\cart(G) = 3$}
 %   \label{fig:immersion}
%\end{figure}

%\begin{proof}
 %   \cite[Example 2.8]{echavarria2021scramble} tells us that $G$ and $H$ in figure \ref{fig:immersion} have scramble numbers 3 and 2 respectively, where $G$ is obtained as an immersion minor of $H$ by first performing a lift along $a, c, d$ and then along $a, d, e$.  By corollary \ref{cart=sn=dsn}, this immediately implies that $\cart(G) = 3 > 2 = \cart(H)$, and so the result follows immediately. 
%\end{proof}

\section{Bounds on Carton Number} \label{bounds-on-carton-number}

In this section we first prove a lower bound on carton number that allows us to find explicit graphs with carton number strictly larger than $n$. Then, we exhibit an exponential lower bound on carton number for graphs of bounded degree and large scramble number.

\begin{proof}[Proof of Theorem \ref{general-lower-bound}]
Let $G = (V, E)$  be a graph such that $\Delta(G) < \sn(G)$.  Take a maximum order scramble $\mathcal S $ of $G$, so that $|| \mathcal S || = \sn(G)$.  As $\Delta(G) < \sn(G)$, we know that $\mathcal S $  does not contain any egg consisting of a single vertex $v$; if it did, then either the at most $\Delta(G)<\sn(G)$ edges incident to $v$ would form an egg-cut giving $||\mathcal{S}||\leq e(\mathcal{S})\leq \Delta(G)<\sn(G)$, or all eggs would contain  and $v$ and thus $||\mathcal{S}||\leq h(\mathcal{S})=1\leq \Delta(G)<\sn(G)$. Thus all eggs of $\mathcal S $ have two or more vertices.  

Next, fix a minimum hitting set of $\mathcal S$ and denote it as $A$.  Note that $|A|=h(\mathcal{S}) \geq \sn(G)$, and if we denote $B = V \setminus A$ then necessarily $|B| \leq |V| - \sn(G)$.  For each vertex $v \in A$, there must exist an egg $e_v \in \mathcal S$ such that $e_v \cap A = \{v\}$: if no such egg $e_v$ existed, then $A \setminus \{v\}$ would be a valid hitting set, but $A$ is a minimum hitting set.  For each $v \in A$, we can thus choose some such corresponding egg $e_v$, and denote the set of $|A|$ such eggs as $\mathcal S_A$.  However, as each egg of $\mathcal S_A$ is of order at least $2$ but has only one of its vertices in $A$, each egg of $\mathcal S_A$ must contain a vertex in $B$.  In particular, $B$ is a valid hitting set of $\mathcal S_A$.  However, we have
\[|B| \leq |V| - \sn(G)\quad\quad\textrm{and}\quad\quad|\mathcal S_A| = |A| \geq \sn(G). \]
Thus, at least $\sn(G)$ eggs of $\mathcal S$ (namely those in $\mathcal{S}_A)$ can be hit by no more than $|V| - \sn(G)$ vertices (namely those in $B$).  As any minimal hitting set of $\mathcal S$ is of size at least $\sn(G)$, and each additional egg not in $\mathcal S_A$ can only cause the hitting number to increase by at most 1, it follows there must be at least $\sn(G) - (|V| - \sn(G))$ eggs in $\mathcal S$ in addition to those of $\mathcal S_A$ already hittable by $B$.  Thus we have
\begin{align*}
|\mathcal{S}| &\geq |\mathcal S_A| + (\sn(G)- |B|)\\ 
&\geq \sn(G) + (\sn(G) - |B|)\\ 
&=2 \sn(G) - |B|\\ 
&\geq 2 \sn(G) - (|V| - \sn(G))\\
&= 3 \sn(G) - |V|
\end{align*}
completing the proof.
\end{proof}

Note that this bound may not be very good for many graphs: we have assumed that each egg aside from the eggs in $\mathcal S_A$ increases hitting number by one, which in practice may not be achievable! The most certain way to increase hitting number is to add eggs that are disjoint from all other eggs in the scramble, but depending on the specific graph and choice of $A$, there may be few, perhaps even zero, such eggs. Adding eggs that overlap with existing eggs may change the hitting set dramatically, making it difficult to find a stronger bound on how many eggs one must add to increase the hitting set to the desired size.

\begin{example}
    Consider the $4\times 4$ rook's graph $R_{4,4}$, which encodes the moves that a rook can make on a $4\times 4$ chess board.  This graph has $16$ vertices which we arrange in a $4\times 4$ grid, with every vertex connected to all other vertices in the same row or column.  Thus $|V(R_{4,4})|=16$ and $\Delta(R_{4,4})=6$. It was shown in \cite{rooks_gonality} that $\sn(R_{4,4})=11$.  Applying Theorem 1.1, we have that 
    \[\cart(R_{4,4})\geq 3\sn(R_{4,4})-|V(G)|=33-16=17.\]
    This is our first explicit example of a graph whose carton number is strictly larger than its number of vertices.
\end{example}

We now give a necessary condition for scrambles to require exponentially many eggs in the number of vertices. Our arguments closely follow those of Grohe and Marx in \cite{exponential_scramble_size}.

\begin{theorem} \label{exponential-carton-number}
    Let $G = (V, E)$ with $\Delta(G) = d$ and $n=|V(G)|$, and let $\mathcal{S}$ be a scramble on $G$ such that $||\scramble|| > \lceil c \cdot n^{1/2 + \epsilon}\rceil $ for some $c > 1$ and $\epsilon > 0$. Then, $|\mathcal{S}| \geq \text{exp}(c \cdot n^{\epsilon} / d)$.
\end{theorem} 
\begin{proof}
    Suppose $A$ is an egg of $\mathcal{S}$ and has size at most $\frac{c \cdot n^{1/2}}{d}$. Then either
    \begin{enumerate} 
        \item There exists a set $A' \in \mathcal{S}$ such that $A' \cap A = \emptyset$. In this case, $E(A,A^C)$ forms an egg-cut.  Since every vertex of $A$ is adjacent to at most $d$ elements of $A^C$, the size of this egg-cut is at most $d\cdot |A|\leq c \cdot n^{1/2}$. Thus, $||\mathcal{S}|| \leq c \cdot n^{1/2}$, contradicting the assumption that $||\mathcal{S}||> \lceil c \cdot n^{1/2 + \epsilon} \rceil$.
        
        \item No such $A' \in \mathcal{S}$ exists. Then, $A$ is a hitting set of $\mathcal{S}$. By assumption, $|A| \leq \frac{c \cdot n^{1/2}}{d}$. Thus, $||\mathcal{S}|| \leq \frac{c \cdot n^{1/2}}{d}$, contradicting the assumption that $||\mathcal{S}||>\lceil c \cdot n^{1/2 + \epsilon} \rceil$.
    \end{enumerate}
    Therefore, we may assume that every $A \in \mathcal{S}$ has cardinality at least $\frac{c \cdot n^{1/2}}{d}$. We now bound the probability that a randomly selected set of vertices does not form a hitting set for $\scramble$. Let $\ell = \lceil n^{1/2 + \epsilon} \rceil$ and $v_1, \ldots , v_{\ell}$ be vertices chosen independently and uniformly at random.  Then,    \begin{align*}
        \Pr(v_i\in A) &= \frac{|V(A)|}{|V(G)|} \geq \frac{ \frac{c \cdot n^{1/2}}{d}}{n} = \frac{c}{n^{1/2} \cdot d}
    \end{align*}
    Furthermore,
    \begin{align*}
        \Pr \left( \forall i: v_i\notin A \right) &= \Pi_{i=1}^{\ell} \left( 1-\Pr(v_i\in A) \right) \leq \left( 1 - \frac{c \cdot n^{-1/2}}{d} \right)^{\ell} \\
        &\leq \exp\left(-\frac{c \cdot n^{-1/2}}{d} \cdot l \right) \leq \exp \left(-\frac{c \cdot n^{\epsilon}}{d} \right)
    \end{align*}
    Lastly, 
    \begin{align*}
        \Pr\left( \{v_1, \ldots , v_l\} \text{ not hitting set for } \scramble \right) &= \Pr \left( \exists A \in \scramble\textrm{ s.t. }(\forall i: v_i\notin A) \right) \\
        &\leq \sum_{A \in \scramble} \Pr \left( \forall i: v_i\notin A \right) \\
        &\leq  |\scramble| \cdot \exp \left( -\frac{c \cdot n^{\epsilon}}{d} \right).
    \end{align*}
    Because $\ell \leq  \lceil c \cdot n^{1/2 + \epsilon}\rceil< \|\scramble\| \leq h(\scramble)$, it follows that no set of $\ell$ vertices can ever hit $\scramble$. Therefore, \\ $\Pr\left( \{v_1, \ldots , v_l\} \text{ is not hitting set for } \scramble \right) = 1$ and 
    \begin{align*}
        |\scramble| \geq \exp \left( \frac{c \cdot n^{\epsilon}}{d} \right).
    \end{align*}
    
\end{proof}

Before we apply this result to prove Theorem \ref{expander-graphs}, we present a quick corollary of this result.

\begin{corollary}\label{corollary:minor_closed_bounded_deg}
Let $\mathcal{G}$ be a family of graphs with bounded maximum degree and only polynomially many subgraphs.  A graph $G\in \mathcal{G}$ with $n$ vertices satisfies $\sn(G)=O(\sqrt{n})$.
\end{corollary}
\begin{proof}
Any scramble $\mathcal{S}$ on such a graph has polynomially many eggs, since each egg is a connected subgraph.  By the contrapositive of the previous theorem, its order must be $O(\sqrt{n})$.
\end{proof}
It is natural to then ask then which graph families satisfy the property of having only polynomially many connected subgraphs.  For minor closed families, this is discussed in \cite{poly-eggs}, where an argument is presented that these are precisely the minor-closed graph families that exclude at least one star graph $K_{1,m}$.  These families will eventually be subsumed in our Corollary \ref{corollary:no_kr_minor_sn_sqrt}, so we do not go into further detail here.

%\begin{proof}
%    Lemma \ref{eppstein-appendix} in the appendix shows that minor closed families with bounded degree have polynomially many connected subgraphs (indeed, these are precisely the minor closed families with polynomially many connected subgraphs!). The contrapositive of Theorem \ref{exponential-carton-number} then implies that these families of graphs have scramble number $O(\sqrt{n})$. 
%\end{proof}

We now show that a family of $d$-regular graphs considered by Grohe and Marx in \cite{exponential_scramble_size} has carton number exponential in $\sqrt{n}$. We first recall the relevant definitions and theorems here. The \emph{vertex expansion of $G$ with parameter $\alpha \in [0,1]$} is defined to be 
    \begin{align*}
        \mathrm{vx}_{\alpha}(G) = \min_{\substack{X \subseteq V(G) \\ 0 < |X| \leq \alpha \cdot |V(G)|} } \frac{|S(X)|}{|X|},
    \end{align*}
 where $S(X)$ is the neighborhood of $X$, i.e. the set of vertices in $V\setminus X$ adjacent to at least one vertex in $X$. 

\begin{proposition} \label{tw-vx}
    \cite[Proposition 1]{exponential_scramble_size} Let $n \geq 1$ and $0 \leq \alpha \leq 1$. Then for every $n$-vertex graph $G$ we have 
    \begin{align*}
        \tw(G) \geq \lfloor \mathrm{vx}_{\alpha}(G) \cdot (\alpha / 2) \cdot n \rfloor.
    \end{align*}
\end{proposition}

\begin{proposition} \label{d-reg-vx}
    \cite{expander-graph-applications} Let $d \geq 3$. Then, for every $\epsilon > 0$ there exists an $\alpha > 0$ and a family $(G_n)_{n \geq 1}$ of $d$-regular graphs for which 
    \begin{align*}
        \mathrm{vx}_{\alpha}(G) \geq d - 1 -\epsilon.
    \end{align*}
\end{proposition}

These results directly lead to Theorem \ref{expander-graphs}.

\begin{proof}[Proof of Theorem \ref{expander-graphs}]
    Combined with Proposition \ref{tw-vx}, Proposition \ref{d-reg-vx} shows that for every $\epsilon > 0$, there exists an $\alpha > 0$ and a family $(G_n)_{n \geq 1}$ of $d$-regular graphs for which $\tw(G) = \lfloor (d-1-\epsilon) \cdot (\alpha/2) \cdot n \rfloor = \Omega(n)$. Because $\tw(G) \leq \sn(G)$, this same family has $\sn(G) = \Omega(n)$. Thus, every maximum order scramble $\scramble$ of $G$ must have $||\scramble|| = \Omega(n)$. Theorem \ref{exponential-carton-number} then shows the second claim. 
\end{proof}

This result implies that carton number can increase dramatically when taking a subgraph.  Consider a complete graph $H$ that is formed by adding missing edges to such an expander graph $G$ on $n$ vertices. $H$ has carton number $n-1$ whereas $G$ has carton number $2^{\Omega(\sqrt{n})}$. Because such a $G$ exists for all $n \geq 1$, the gap in carton numbers of $H$ and subgraph $G$ can grow arbitrarily large. \\ 

We close this section by showing that there exist efficient algorithms to find large scrambles on graphs.

\begin{proposition}
    Let $k \geq 2$ and let $G$ be a graph  of treewidth greater than $3k$. Then $G$ has a scramble of order $\Omega(\frac{\sqrt{k}}{\ln(k)})$ and size $O(k^{3/2} \ln(n))$.  
\end{proposition}
\begin{proof}
    The analogous result for brambles is proved in  \cite[Lemma 14]{exponential_scramble_size}. Taking such a bramble $\mathcal B$, Proposition \ref{scramble-order-of-brambles} then implies that $\mathcal{B}$ is a scramble of order $\Omega(\frac{\sqrt{k}}{\ln(k)})$ and size $O(k^{3/2} \ln(n))$.
\end{proof}

\begin{proposition}
    There exists a randomized polynomial time algorithm that, given a graph $G$ of treewidth $k$, constructs with high probability a scramble in $G$ of order $\Omega(\frac{\sqrt{k}}{\ln^3(k)})$ and size $O(k^{3/2} \ln k \ln n).$
\end{proposition}
\begin{proof}
     The analogous result for brambles is proved in \cite[Theorem 3.1]{scramble_construction}. Taking such a bramble $\mathcal B$, Proposition \ref{scramble-order-of-brambles} then implies that $\mathcal{B}$ is a scramble of order $\Omega(\frac{\sqrt{k}}{\ln^3(k)})$ and size $O(k^{3/2} \ln k \ln n)$. Therefore, the same algorithm shows the claim.
\end{proof}

\section{Properties of disjoint scramble number} \label{properties-disjoint-scramble-number}  In this section we focus on properties of disjoint scramble number and its relationship to other graph invariants, with the eventual goal of using it to compute carton numbers via Theorem \ref{cart=sn}.   %of disjoint scramble number and its relationship to treewidth. %Furthermore, by studying cases when $\dsn(G) = \gon(G)$, we also compute the carton number of classes of Cartesian product graphs. 

\iffalse

\begin{proposition}
    A graph $G$ has $\dsn(G) \geq k$ if and only if there exists a $k$-partition with egg-cut $\geq k$.  
\end{proposition}
\begin{proof}
    \textcolor{red}{To be completed} 
\end{proof} 

\fi

\begin{proposition}
 Disjoint scramble number is subgraph monotone, and invariant under subdivision and smoothing.
\end{proposition}
\begin{proof} 
If $G'$ is a subgraph of $G$, then any disjoint scramble $\mathcal{S}$ on $G'$ is also a disjoint scramble on $G$.  As argued in the proof of \cite[Proposition 4.5]{new_lower_bound}, the order of $\mathcal{S}$ on $G$ is at least as large as the order of $\mathcal{S}$ on $G'$.  Choosing $\mathcal{S}$ disjoint of maximal order on $G'$ gives $\dsn(G)\geq \dsn(G')$.

We will argue that disjoint scramble number is invariant under subdivision; it then follows that it is invariant under smoothing.  By Proposition \ref{dsn=sn}, if $G'$ is a subdivision of $G$ and $\dsn(G)\leq 2$ or $\dsn(G')\leq 2$ we have $\dsn(G)=\sn(G)=\sn(G')=\dsn(G')$ by a combination of \cite[Proposition 5.2]{sn_two} and Corollary \ref{cart=sn=dsn}; thus for the remainder we may take $\dsn(G)\geq 3$ and $\dsn(G')\geq 3$.  It suffices by induction to consider the case that a subdivision $G'$ of $G$ is obtained by subdividing a single edge between adjacent vertices $u$ and $w$, introducing a new vertex $v$ between them.

 We first show $\dsn(G')\geq \dsn(G)$.  Let $\mathcal{S}$ be a disjoint scramble of order $\dsn(G)$ on $G$. Construct the scramble $\mathcal{S}'$ on $G'$ from Proposition \ref{prop:specific_scrambles_subdivision}(i). The construction of $\mathcal{S}'$ ensures it is still disjoint, and $\dsn(G')\geq ||\mathcal{S}'||\geq ||\mathcal{S}||=\dsn(G)$.

We now show $\dsn(G)\geq \dsn(G')$.  Let $\mathcal{S}'$ be a disjoint scramble of order $\dsn(G')\geq 3$ on $G'$. Construct the scramble $\mathcal{S}$ on $G$ from Proposition \ref{prop:specific_scrambles_subdivision}(i). The construction of $\mathcal{S}$ ensures it is still disjoint, and $\dsn(G)\geq ||\mathcal{S}||\geq ||\mathcal{S}'||=\dsn(G')$.

We conclude that $\dsn(G)=\dsn(G')$, as desired.
\end{proof}

\begin{proposition}
    Let $G$ be a $k$-vertex-connected graph. Then, $\dsn(G) \geq k$.
\end{proposition}
\begin{proof}
    Consider the \textit{vertegg scramble} $\scramble$, consisting of all one-vertex eggs. Note that $||\scramble|| = \min(|V(G)|, \lambda(G) )$. Since $G$ is $k$-vertex-connected, it follows that $|V(G)| \geq k$ and that $\lambda(G) \geq k$ (since $\kappa(G) \leq \lambda(G)$). Therefore, $\dsn(G) \geq ||\scramble|| = \min(|V(G)|, \lambda(G)) \geq k$. 
\end{proof}

\begin{proposition} 
    Let $G$ be in a graph class for which the maximum degree grows as $d(n)$. Then, $\dsn(G) = O(\sqrt{n \cdot d(n)})$.
\end{proposition}
\begin{proof}
    Consider any $k$-partition of the vertices of $G$ into connected subgraphs. We claim that the maximum order of such a disjoint scramble is at most $\min(k, d(n) \cdot \frac{n}{k})$. To see this, note that the hitting number of such a scramble must be $k$. Thus, maximizing scramble order reduces to maximizing its egg-cut. Because $G$ has degree bounded by $d(n)$, the minimum egg-cut is bounded above by $d(n) \cdot \min_{E \in \scramble} |E|$. Observe that $\min_{E \in \scramble} |E| \leq \frac{n}{k}$. It then follows that
    \begin{align*}
        \dsn(G) &\leq \max_{k \in [n]} \min\left(k, d(n) \cdot \frac{n}{k}\right).
    \end{align*}
    Because $d(n) \cdot \frac{n}{k}$ decreases monotonically with $k$, the maximum is achieved when $k = d(n) \cdot \frac{n}{k}$. Thus, $k = \sqrt{n \cdot d(n)}$ and 
    \begin{align*}
        \dsn(G) &\leq \max_{k \in [n]} \min\left(k, d(n) \cdot \frac{n}{k}\right) \\
        &= O(\sqrt{n \cdot d(n)}).
    \end{align*}
\end{proof}

\begin{corollary} \label{dsn-sqrt-n}
    If $G$ is a graph with bounded degree $d$, $\dsn(G) = O(\sqrt{n})$.
\end{corollary}

We now explore the relationship between disjoint scramble number and treewidth. We already know that $\tw(G) \leq \sn(G) \leq \gon(G)$. Proposition \ref{dsn-sn-cart-inequality-chain} shows that $\dsn(G) \leq \sn(G) \leq \gon(G)$.  We remark that for multigraphs $G$, it is easy to find large gaps between $\dsn(G)$ and $\tw(G)$; for instance, taking $G$ to be a tree on $n$ vertices where each edge is repeated $n$ times gives an example with $\tw(G)=1$ and $\dsn(G)=n$.
 Thus we restrict our study to simple graphs and show that the relationship between $\dsn(G)$ and $\tw(G)$ can still vary. A class of graphs known as $k$-trees helps to establish the relationship in one direction.

We define
   \emph{$k$-trees} to be maximal simple graphs with treewidth $k$. By \cite{On-the-structure-of-k-trees}, they can be characterized as follows:
    \begin{itemize}
        \item The complete graph on $k+1$ vertices is a $k$-tree.
        \item A $k$-tree $G$ can be obtained from a $k$-tree $H$ by connecting a new vertex to exactly $k$ vertices in $H$ that form a $k$-clique.
        \item No other graph is a $k$-tree.
    \end{itemize}

Immediately, the definition implies that for every $k$, there exists a $k$-tree for which $\dsn(G) = \tw(G)$, namely the complete graph $K_{k+1}$. We can also use $k$-trees to construct simple graphs for which $\dsn(G) > \tw(G)$.

\begin{proposition}
    For every $k$, there exists a $k$-tree $G$ for which $\dsn(G)-\tw(G)\geq k-1$.
\end{proposition}
\begin{proof}  Let $G$ be the simple graph with vertex set $v_1,\ldots,v_{4k}$ where $v_i$ is adjacent to $v_j$ if and only if $|i-j|\leq k$. This graph is illustrated for $k=3$ in Figure \ref{figure:3_tree}. Note that $G$ can be obtained from $K_{k+1}$ by iteratively adding vertices connected to $k$-cliques, so $G$ is a $k$-tree and $\tw(G)=k$.

\begin{figure}[hbt]
    \centering
   \includegraphics[scale = 1]{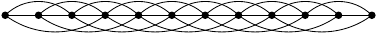}
    \caption{A $3$-tree $G$ with $\dsn(G)\geq 5$, achieved by having single vertex eggs on the middle six vertices}
   \label{figure:3_tree}
\end{figure}

Let $\mathcal{S}$ be a scramble consisting of verteggs, one for each of $v_{k+1}$, $v_{k+2},\ldots,v_{3k}$.  Since there are $2k$ disjoint eggs, we have $h(\mathcal{S})=2k$.

To lower bound $e(\mathcal{S})$, we will find a collection of edge disjoint paths connecting an arbitrary pair of eggs, say $v_i$ and $v_j$ with $k+1\leq i<j\leq 3k+1$.  Letting $1\leq m\leq k$, start our path as $v_i-v_{i+m}$.  From there, increase the index of the vertex by $k$ at a time, until we fall into $\{v_{j},v_{j+1},\ldots,v_{j+k-1}\}$ (note that this may take zero steps of size $k-1$).  Then, move to $v_j$.  This gives a total of $k$ paths from $v_i$ to $v_j$, and they are pairwise disjoint.  For another collection of paths, let $0\leq m\leq k-2$, and start the path $v_i-v_{i-m}$.  From there, increase the index of the vertex by $k-1$ at a time, until we fall into $\{v_{j-k+1},v_{j-k+2},\ldots,v_{j}\}$ (again, this may take zero steps). This gives us another $k-1$ paths, which are pairwise edge-disjoint not only from one another, but also from the first $k$ paths.  Thus we have found a set of $2k-1$ pairwise edge-disjoint paths connecting an arbitrary pair of eggs.  Any egg cut must include at least one edge from each of these paths, so $e(\mathcal{S})\geq 2k-1$.

We have $||\mathcal{S}||=\min\{h(\mathcal{S}),e(\mathcal{S})\}\geq\min\{2k,2k-1\}=2k-1$, so $\dsn(G)\geq ||\mathcal{S}||=2k-1$.  We conclude that $\dsn(G)-\tw(G)\geq 2k-1-k=k-1$, as desired.
\end{proof}

On the flip side, $\dsn(G)$ can also be arbitrarily small as compared to $\tw(G)$ for simple graphs. The expander graphs from Theorem \ref{expander-graphs} can be used to show that treewidth (and therefore also scramble number) can grow quadratically in disjoint scramble number. 

\begin{corollary}
    There exists a family of graphs $(G_n)_{n \geq 1}$ for which $\tw(G) = \Omega(\dsn(G)^2)$.
\end{corollary}
\begin{proof}
    Consider the $d$-regular expander graphs of Theorem \ref{expander-graphs} with $\tw(G) = \Omega(n)$. By Proposition \ref{dsn-sqrt-n}, $\dsn(G) = O(\sqrt{n})$. Therefore, $\tw(G) = \Omega(\dsn(G)^2)$. 
\end{proof}

We now turn to chip-firing invariants. 
 Since disjoint scramble number is a lower bound on gonality, it is natural to ask when they are equal. Cartesian products of graphs yield many such examples. We first establish a lower bound on the disjoint scramble number of Cartesian products of graphs.

\begin{proposition} \label{dsn-cartesian-product}
    Let $G$ and $H$ be connected graphs. Then, $\dsn(G \square H) \geq \min(V(H), V(G) \lambda(H))$. 
\end{proposition}
\begin{proof}
    Consider the scramble $\scramble$ whose eggs are canonical copies of $G$. The hitting number of this scramble is the number of eggs, which is $|V(H)|$. To bound the egg-cut number of $\scramble$, consider eggs $E_1, E_2 \in \scramble$ with $A \subset V(G \square H)$ such that $E_1 \subset A$ and $E_2 \subset A^C$. By construction, $E_1 = G \square w_1$ and $E_2 = G \square w_2$ for some $w_1, w_2 \in V(H)$. Note that for any $v \in G$, there exist at least $\lambda(H)$ edge-disjoint paths between $v \square w_1$ and $v \square w_2$. Because each edge-disjoint path must have at least one edge crossing the cut $(A, A^C)$, we conclude that the egg-cut is at least $|V(G)| \lambda(H)$. Therefore, $\dsn(G \square H) \geq \min(V(H), V(G) \lambda(H))$. 
\end{proof}

\begin{proposition} 
    Let $G$ and $H$ be connected graphs.
    \begin{enumerate}[label = (\roman*)]
        \item If $G$ is a tree and $H$ is a graph with $\frac{|V(H)|}{\lambda(H)} \leq |V(G)|$, then $\dsn(G \square H) = \sn(G \square H) = \cart(G \square H) = \gon(G \square H) = |V(H)|$.
        \item If $G$ and $H$ are graphs with $\gon(H) = \lambda(H)$ and $|V(G)| \leq \frac{|V(H)|}{\lambda(H)}$, then $\dsn(G \square H) = \sn(G \square H) = \cart(G \square H) = \gon(G \square H) = |V(G)| \cdot \lambda(H)$. 
    \end{enumerate}
    
\end{proposition}
\begin{proof} 
    Proposition \ref{dsn-cartesian-product} gives that $\dsn(G \square H) \geq |V(H)|$ and $\dsn(G \square H) \geq |V(G)| \lambda(H)$ for (i) and (ii) respectively. \cite[Theorem 5.1(i)]{echavarria2021scramble} combined with Propositions \ref{dsn-sn-cart-inequality-chain} and \ref{cart=sn} then show (i). Similarly, \cite[Theorem 5.1(ii)]{echavarria2021scramble} combined with Propositions \ref{dsn-sn-cart-inequality-chain} and \ref{cart=sn} show (ii).
\end{proof} 

From this, we get strengthened versions of \cite[Corollaries 5.2-5.6]{echavarria2021scramble}. The table below summarizes notable graphs for which the above Theorem applies, giving $\dsn(G \square H) = \sn(G \square H) = \cart(G \square H) = \gon(G \square H)$. In the table, $T$ denotes a tree; $P_\ell$ denotes a path on $\ell$ vertices; $K_\ell$ denotes a complete graph on $\ell$ vertices; and $K_{m,n}$ denotes a complete bipartite graph on $m$ and $n$ vertices.  For all graphs, we assume $|V(G)|, |V(H)| \geq 2$.

\begin{table}[H]\label{table1}
\centering
    \begin{tabular}{|c c c c|}
        \hline
        $G$ & $H$ & Assumptions & $\cart(G \square H)$ \\
        \hline
        $T$ & $H$ & $\gon(H) = \lambda(H) = k$ & $\min(|V(H)|, k|V(G)|)$ \\
        % $T$ & $H$ & $\frac{|V(H)|}{\lambda(H)} \leq |V(G)|$ & $|V(H)|$ \\
        % $G$ & $H$ & $\gon(H) = \lambda(H) = k, |V(G)| \leq \frac{|V(H)|}{\lambda(H)}$ & $k|V(G)|$ \\
        $P_{\ell}$ & $P_{m} \square P_{n}$ & $\ell \geq mn/2$ & $mn$ \\
        $G$ & $K_{\ell} \square T$ & $|V(G)| \leq |V(T)|, |V(T)| \geq 2$ & $\ell |V(G)|$ \\
        $G$ & $H$ & $\lambda(H) = \gon(H) = 2, |V(G)| \leq |V(H)|/2$ & $2|V(G)|$ \\
        $G$ & $K_{m,n}$ & $|V(G)| \leq (m+n)/m, m \leq n$ & $m|V(G)|$ \\
        \hline
    \end{tabular} \\
    \caption{\label{dsn-table} Cartesian Products of graphs for which $\dsn(G \square H) = \sn(G \square H) = \cart(G \square H) = \gon(G \square H)$}
\end{table}

In fact, \cite[Corollary 5.8]{screewidth-og} shows equality with $\scw(G \square H)$ in many of these cases, sometimes with added conditions. With the existing assumptions, cases 3 and 5 also have equality with $\scw(G \square H)$. In cases 1 and 4, adding the assumption that $\lambda(H) = \scw(H)$ also gives equality with $\scw(G \square H)$. In case 2, we note that the screewidth of the 3-dimensional grid graph $G_{l,m,n}$ where $\ell \geq mn/2$ is also $mn$. \\

We now show more cases where all five invariants converge by strengthening \cite[Corollary 5.4]{screewidth-og} in two stages.

\begin{proposition} \label{5-equality}
    If $G$ is a $k$-edge-connected graph of gonality $k$, then $\dsn(G) = \sn(G) = \cart(G) = \scw(G) = \gon(G) = k$.
\end{proposition} 
\begin{proof} 
    \cite[Corollary 5.4]{screewidth-og} shows equality $\sn(G) = \scw(G) = \gon(G) = k$. It then suffices to show that $\dsn(G) = \sn(G) = \cart(G)$. Recall that the {vertegg scramble} on $G$, where every vertex of $G$ constitutes an egg, shows that $\dsn(G) \geq \min(|V(G)|, \lambda(G)) = \min(|V(G)|, k)$. Because $k = \gon(G) \leq |V(G)|$, it follows that $\dsn(G) \geq k$. Propositions \ref{dsn-sn-cart-inequality-chain} and \ref{cart=sn} then show the claim.
\end{proof}

We can use this result to compute the carton number of common families of graphs:  the cycle graphs $C_n$, with $n$ vertices arranged in a cycle; the complete multipartite graphs $K_{n_1, n_2, \ldots , n_k}$, with vertex sets of size $n_i$ where each vertex is adjacent to all vertices not in its set; the grid graphs $G_{m,n}$, consisting of $m\times n$ vertices arranged in a grid; and the stacked prism graphs $Y_{m,n}$, constructed from the grid graph $G_{m,n}$ by adding edges between the top and bottom rows of $n$ vertices.

\begin{proposition}
    For the following graphs families, we have $\dsn(G) = \sn(G) = \cart(G) = \scw(G) = \gon(G)$. 
    
    \begin{enumerate}[label=(\arabic*)]
        \item Cycle graphs $C_n$, with $\cart(C_n) = 2$.
        \item Complete multipartite graphs $K_{n_1, n_2, \ldots , n_k}$, with $\cart(K_{n_1, n_2, \ldots , n_k}) = \sum_i n_i - \max_i\{n_i\}$.
        \item Grid graphs $G_{m,n}$, with $\cart(G_{m,n}) = \min(m,n)$.
        \item The stacked prism graphs $Y_{m,n}$, with $\cart(Y_{m,n}) = \min(m, 2n)$.
    \end{enumerate}
\end{proposition}
\begin{proof}
    The cycle graph $C_n$ is $2$-edge-connected and has gonality $2$. The $k$-partite graph $K_{n_1, n_2, \ldots , n_k}$ is an $\ell$-edge-connected graph with gonality $\ell$, where $\ell = \sum_i n_i - \max_i\{n_i\}$. Proposition \ref{5-equality} immediately proves claims (1) and (2). 
    
  By  \cite[Proposition 5.5]{screewidth-og}, we have $\sn(G_{m,n}) = \scw(G_{m,n}) = \gon(G_{m,n}) = \min(m,n)$. The disjoint scramble of all rows or all columns (whichever is smaller), shows that $\dsn(G_{m,n}) \geq \min(m,n)$, thereby proving (3). By \cite[Proposition 5.5]{screewidth-og}, we have $\sn(Y_{m,n}) = \scw(Y_{m,n}) = \gon(Y_{m,n}) = \min(m,2n)$. Notice that $Y_{m,n} = C_m \square P_n$, so case $1$ in Table \ref{dsn-table} shows that $\dsn(Y_{m,n}) = \cart(Y_{m,n}) = \min(m, 2n)$, thereby proving (4). 
\end{proof}

%To generalize even further, we use a result from \cite{cartesian-product-edge-connectivity} on the edge connectivity of the cartesian product of graphs.

%\begin{lemma}[Theorem 3, \cite{cartesian-product-edge-connectivity}] Given any two graphs $G, V$
    
%\end{lemma}
%, which will allow us to prove a statement analogous to corollary 5.2 from \cite{echavarria2021scramble} with an added condition.

\begin{proposition}
    If $G$ is a tree, $H$ is $k$-edge-connected with $\gon(H) = k$, and $1 + \delta(H) \geq \min(|V(H)|, |V(G)| \lambda(H))$, then $\dsn(G \square H) = \sn(G \square H) = \cart(G \square H) = \scw(G \square H) = \gon(G \square H) = \min(|V(H)|, |V(G)| \lambda(H))$.
\end{proposition}
\begin{proof}
    %Given $G$ and $H$ as above, by Theorem 3 from \cite{cartesian-product-edge-connectivity} and the fact $G$ is a tree necessarily $\lambda(G \square H) = \min(\delta(G) + \delta(H), \lambda(G), |V(H)|, \lambda(H), |V(G)|) = \min(1 + \delta(H),  |V(H)|, \lambda(H), |V(G)|) $  

    By \cite[Corollary 5.2]{echavarria2021scramble} (row 1 of table 1), we have that $\dsn(G \square H) = \sn(G \square H) = \cart(G \square H) = \gon(G \square H)$. \cite{cartesian-product-edge-connectivity} proves that $\lambda(G \square H) = \min(\delta(G) + \delta(H), |V(G)|\lambda(H), |V(H)| \lambda(G))$. Because $G$ is a tree, $\lambda(G) = \delta(G) = 1$. Therefore, when $1 + \delta(H) \geq \min(|V(H)|, |V(G)| \lambda(H))$, we have that \\
    $\lambda(G \square H) = \min(|V(H)|, |V(G)|\lambda(H)) = \gon(G \square H)$. By \cite[Lemma 10.24]{sandpiles} and \cite[Theorem 5.3]{screewidth-og}, we then get that $\scw(G \square H) \leq \gon(G \square H)$. Because $\sn(G \square H) = \gon(G \square H)$, it then follows that $\scw(G \square H) = \gon(G \square H)$.  
\end{proof}

\section{Complexity of scramble number and variants} \label{complexity-sn-variants}

We now study the computational complexity of scramble number and disjoint scramble number. First, we apply Courcelle's Theorem to show that the decision problem $\dsn(G) \geq k$ is fixed parameter tractable when parametrized by treewidth and $k$. We then show that the scramble number and gonality of specific graph families can be efficiently approximated to a constant factor.

\subsection{Fixed parameter tractability of disjoint scramble number}

Courcelle's theorem is a logic-based meta-theorem that gives a sufficient condition for graph-theoretic properties to be decided in fixed-parameter linear time. In particular, a problem can be decided in fixed parameter linear time if it can be solved in time $f(k) \cdot n$ for some computable function $f$ and $k$ is the variable we parameterize. Courcelle's theorem relies on a subset of second-order logic on graphs known as monadic second-order logic, which we recall here; see \cite{fundamentals_of_parameterized_complexity} for more details.

\textit{Monadic second-order logic} (MS$_2$) allows logical elements $\land$, $\lor$, $\neg$, $\forall$, $\exists$, and variables representing vertices, edges, sets of vertices, and sets of edges. The following binary relations are also allowed
    \begin{enumerate}[label=(\arabic*)]
        \item $v \in V$ for vertex variable $v$ and vertex set $V$
        \item $e \in E$ for edge variable $e$ and edge set $E$
        \item $\mathrm{inc}(e,v)$ for edge variable $e$, vertex variable $v$, indicating that edge $e$ is incident to vertex $v$
        \item $\mathrm{adj}(u,v)$ for vertex variables $u$ and $v$, indicating that vertices $u$ and $v$ are adjacent
        \item Equality between vertices, edges, sets of vertices, or sets of edges
    \end{enumerate} 

\begin{theorem}[Courcelle's Theorem, Theorem 4.1.1 in \cite{courcelle's-theorem-overview-and-applications}] \label{courcelle}  Let $\varphi$ be a statement about hypergraphs written in monadic second order logic. Then there is an algorithm which, for any hypergraph $G$ on $n$ vertices of treewidth $\tw(G)$, decides whether $\varphi$ holds in $G$ in time $O(f(|\varphi| , \tw(G)) \cdot n)$, for some computable function $f$. 
\end{theorem}

\begin{lemma}
    \cite[Corollary 6.1.2]{courcelle's-theorem-overview-and-applications} $\mathrm{Path}(P, x, y)$, for vertices $x, y \in V(G)$ and subset $P\subseteq E(G)$ expresses whether $x$ and $y$ are connected in $G$ using edges in $P$. It can be expressed as a constant-size formula in $\mathrm{MS}_2$.
\end{lemma}

\begin{lemma} \label{properties-mso2}
    The following graph properties can be expressed in $\mathrm{MS}_2$:

    \begin{enumerate}[label=(\arabic*)]
        \item $\mathrm{Path}(P, X, Y)$, for subsets $X, Y \subseteq V(G)$ and $P\subseteq E(G)$, which expresses whether there is a path between some vertex of $X$ and some vertex of $Y$ using edges in $P$.
        \item $\mathrm{Disjoint}(E_1, E_2)$, for subsets $E_1, E_2 \subseteq E(G)$, which expresses whether the two edge sets are disjoint.
        \item $\mathrm{Connected}(U)$ for subset $U \subseteq V(G)$, which expresses whether the vertex set $U$ is connected.
        \item $\mathrm{Partition}(U_1, \ldots , U_k)$ for subsets $U_1, \ldots , U_k \subseteq V(G)$ ,which expresses whether vertex sets $U_1, \ldots , U_k$ are disjoint and completely cover $V(G)$. 
    \end{enumerate} 
\end{lemma}

\begin{proof}
    We express them as follows:

    \begin{enumerate}[label=(\arabic*)]
        \item $\mathrm{Path}(P, X, Y) = \exists x, y \left[ x \in X \land y \in Y \land \mathrm{Path}(P, x, y) \right]$
        
        \item $\mathrm{Disjoint}(E_1, E_2) = \forall e \left[e \in E_1 \to e \notin E_2 \right]$
        
        \item $\mathrm{Connected}(U) = \forall x,y \left[ x \in U \land y \in U \to \exists P \textrm{ s.t. }   \mathrm{Path}(P, x, y) \land \forall e \left[e \in P \to [\mathrm{inc}(e, z) \to z \in U \right] \right]$
        
        \item $\mathrm{Partition}(U_1, \ldots , U_k) = \forall x
        \left[ \bigvee_{i=1}^{k} x \in U_i \right] \land \neg \exists y \textrm{ s.t. } \left[ \bigvee_{i \neq j}^{k} (y \in U_i \land y \in U_j) \right]$
    \end{enumerate}

    We also note that formulae for Path, Disjoint, and Connected have constant size. The formula for Partition has size $O(k^2)$.  
\end{proof}

By a simple invocation of Courcelle's theorem, we can then show Theorem \ref{dsn-fpt}.

\begin{proof}[Proof of Theorem \ref{dsn-fpt}]  First we argue that any graph $G$ has a disjoint scramble of order $\textrm{dsn}(G)$ such that the eggs partition $V(G)$. Indeed, let $\scramble$ be a disjoint scramble with $||\scramble||=\dsn(G)$.  If the eggs of $\scramble$ partition $V(G)$, we are done.  If not, choose a vertex $v$ not in any egg that is adjacent to a vertex in some egg $E\in \scramble$, and replace $E$ with $E\cup \{v\}$.  As $G$ is connected, we can repeat this process until all vertices have been included, giving a disjoint scramble $\mathcal{S}'$ with eggs partitioning $V(G)$.  Note that $h(\mathcal{S}')=|\scramble'|=|\scramble|=h(\mathcal{S})$, and that any egg-cut for $\mathcal{S}'$ is also an egg-cut for $\mathcal{S}$, yielding $e(\mathcal{S})\leq e(\mathcal{S'})$.  It follows that $||\mathcal{S}'||\geq ||\mathcal{S}||=\dsn(G)$, and so $||\mathcal{S}'||=\dsn(G)$.

    This allows us to express $\dsn(G) \geq k$ in $\mathrm{MS}_2$ as follows:
    \begin{align*} 
        \exists U_1, \ldots , U_k \vast[ \mathrm{Partition}(U_1, \ldots , U_k) \land  \bigwedge_{i=1}^{k} \mathrm{Connected}(U_i)
        \land \bigwedge_{i \neq j}^{k} \left[ \exists P_1, \ldots , P_k \left[ 
        \bigwedge_{l=1}^{k} \mathrm{Path}(P_l, U_i, U_j) \land \bigwedge_{s \neq t}^{k} \mathrm{Disjoint}(P_s, P_t) \right] \right] \vast] 
    \end{align*}
    By Lemma \ref{properties-mso2}, each graph property used has size $h(k)$ for some computable $h$. Therefore, we have$|\varphi| = g(k)$ for some computable $g$. By Theorem \ref{courcelle}, there exists an algorithm that decides whether $\dsn(G) \geq k$ in time $O(f(|\varphi|, \tw(G)) \cdot n ) = O(f(g(k), \tw(G)) \cdot n)$ for some computable function $f$. 
\end{proof} 

Our study of the computational complexity of these problems naturally leads us to ask about stronger results regarding scramble number and to further answer questions regarding the complexity of both scramble number and disjiont scramble number.

\begin{question}
    Is scramble number fixed parameter tractable?
\end{question}

We remark that a similar application of Courcelle's Theorem is not admissible due to Theorem \ref{exponential-carton-number}, since quantification over exponentially many sets may be necessary. From \cite{sn_two}, we know that there exists a polynomial time algorithm to check whether $\sn(G) \geq k$ for all $k \leq 3$. However, the question remains open in full generality. 

\begin{question}
    Is disjoint scramble number NP-Hard to compute?
\end{question}

From \cite{echavarria2021scramble} we know that scramble number is NP-Hard to compute, but the same is unclear for disjoint scramble number. We conjecture that deciding $\dsn(G) \geq k$ is NP-Hard, thereby making the problem NP-Complete.  

\subsection{Approximability of scramble number and gonality} \label{approximation-section} 

Gijwisjt et al. show in \cite{gijswijt2019computing} that computing graph gonality is APX-Hard. While this is reason to believe that gonality cannot be efficiently approximated to a constant factor in general, there may exist efficient approximation algorithms for specific classes of graphs. The approximability of scramble number is unknown in the general case. We outline a few specific classes of graphs that admit a constant factor approximation for scramble number and gonality. 

\begin{proposition} \label{2-approximation}
    (Gavril's Algorithm \cite[pg. 432]{combinatorial_optimization_algorithms_and_complexity}) Let $G$ be a simple graph on $n$ vertices. Then, $n - \alpha(G)$ (also equal to the size of the \textit{minimum vertex cover} of $G$) can be 2-approximated in polynomial time.
\end{proposition}

Gavril's algorithm is a classic folklore result in approximation algorithms that exploits the trivial relationship between maximal matchings and minimum vertex covers. In fact, we show that the stronger result of Theorem \ref{k-hitting-set} is true when generalized to the $k$-component independence number $\alpha_k(G)$. 

\begin{proof}[Proof of Theorem \ref{k-hitting-set}]
     A $k$-approximation algorithm for the $k$-hitting set problem is constructed in \cite{k-hitting-set-k-approx}. In particular, the paper presents an algorithm $A$ that takes a collection of sets $S$, where each set $s \in S$ has size $|s| \leq k$, and outputs a $k$-approximation of the minimum hitting number. This means that $h(S) \leq A(S) \leq k \cdot h(S)$. We consider the $k$-uniform scramble, $\varepsilon_k$, introduced in \cite{uniform_scrambles}. By \cite[Lemma 3.2]{uniform_scrambles}, $h(\varepsilon_k) = n - \alpha_{k-1}(G)$. Because all sets in $\varepsilon_k$ have size $k$, its hitting number can be $k$-approximated. 
\end{proof}

Theorem \ref{k-hitting-set} subsumes Proposition \ref{2-approximation}. Previous results, namely \cite[Theorem 3.2, Corollaries 3.1 and 3.2, Theorem 3.3]{uniform_scrambles}, give sufficient conditions for $\sn(G) = \gon(G) = n - \alpha_m(G)$ for various values of $m$. This immediately allows us to widen the class of graphs which admit a polynomial time, constant factor approximation algorithm for scramble number and gonality to those from Theorem \ref{gon-sn-approximation}.

\begin{proof}[Proof of Theorem \ref{gon-sn-approximation}]
    (1) By \cite[Corollary 3.1]{uniform_scrambles}, $\sn(G) = \gon(G) = n - \alpha_2(G)$. (2) By \cite[Corollary 3.2]{uniform_scrambles}, $\sn(G) = \gon(G) = n - \alpha_2(G)$. (3) By \cite[Theorem 3.3]{uniform_scrambles}, $\sn(G) = \gon(G) = n - \alpha_3(G)$. (4) By \cite[Theorem 3.2]{uniform_scrambles}, $\sn(G) = \gon(G) = n - \alpha(G)$. Theorem \ref{k-hitting-set} then shows all four claims.
\end{proof}

 Given Theorem \ref{k-hitting-set}, it is natural to ask whether the order of $k$-uniform scrambles approximate scramble number and/or gonality to a constant factor. The answer is more nuanced than it seems, but we now show a sufficient condition under which this is the case. 

\begin{proposition} \label{c-approx}
    Let $G$ be a graph with $\alpha(G) \leq \frac{c-1}{kc-1} \cdot n$ for some $c \in \mathbb{R}^{+}$ and $ k \in \mathbb{N}$ with $c > 1$. Then, $\frac{1}{c}(n - \alpha(G)) \leq n - \alpha_k(G) \leq n - \alpha(G)$. 
\end{proposition}
\begin{proof}
    We will first show that $\frac{1}{k} \alpha_k(G) \leq \alpha(G) \leq \alpha_k(G)$. Note that $\frac{1}{k} \alpha_k(G) \leq \alpha(G)$ because choosing one representative from each component of a $k$-component independent set forms a $1$-component independent set. Secondly, $\alpha_k(G) \leq \alpha(G)$ because any $1$-component independent set is also a $k$-component independent set. It is then easy to see that 
    \begin{align*}  
        n - k\alpha(G) \leq n - \alpha_k(G) \leq n - \alpha(G)
    \end{align*}
    We have assumed that $\alpha(G) \leq \frac{c-1}{kc-1} \cdot n$, therefore $\frac{1}{c} (n - \alpha(G)) \leq n - k \alpha(G)$ by simply rearranging. We can then conclude that
    \begin{align*}
        \frac{1}{c}(n - \alpha(G)) \leq n - k\alpha(G) \leq n - \alpha_k(G) \leq n - \alpha(G)
    \end{align*}
    which shows the claim.
\end{proof}

More simply, Proposition \ref{c-approx} tells us that for graphs with small independence number, the hitting set of the $(k+1)$-uniform scramble is a $c$-approximation of $n - \alpha(G)$, which is an upper bound on the gonality (and therefore scramble number) of $G$  by \cite[Theorem 3.1]{gonality_of_random_graphs}. If $n-\alpha_k(G)$ is indeed the order of the $(k+1)$-uniform scramble, then combined with Theorem \ref{k-hitting-set} we can construct a $(k+1)c$ approximation to both $\sn(G)$ and $\gon(G)$. The following result formalizes this construction. 

\begin{proposition} \label{kc-approx}
    If $G$ is a graph with $\alpha(G) \leq \frac{c-1}{kc-1} \cdot n$ for some $c \in \mathbb{R}^{+}$ and $ k \in \mathbb{N}$ with $c > 1$ and $\lambda_{k+1}(G) \geq n - \alpha_k(G)$, then both $\sn(G)$ and $\gon(G)$ can be $(k+1)c$ approximated in polynomial time. 
\end{proposition}
\begin{proof}
    By Theorem \ref{k-hitting-set}, there exists a polynomial time approximation algorithm $A$ that outputs $(n - \alpha_k(G)) \leq A(G) \leq (k+1)(n - \alpha_k(G))$. Since $\lambda_{k+1} \geq n - \alpha_k(G)$, it follows that for the $k$-uniform scramble $\varepsilon_k$ on $G$, $||\varepsilon_k|| = n - \alpha_k(G) \leq \sn(G) \leq \gon(G) \leq n - \alpha(G)$ \cite[Theorem 3.1]{uniform_scrambles}. We construct algorithm $A'$ by setting $A'(G) = \frac{1}{k+1} \cdot A(G)$. Note then that $\frac{1}{k+1} (n - \alpha_k(G)) \leq A'(G) \leq n - \alpha_k(G)$. \\ \\
    Because $\alpha(G) \leq \frac{c-1}{kc-1} \cdot n$, Proposition \ref{c-approx} implies that $\frac{1}{c}(n - \alpha(G)) \leq n - \alpha_k(G)$. Thus, algorithm $A'$ guarantees $\frac{1}{c(k+1)}(n - \alpha(G)) \leq A'(G) \leq n - \alpha_k(G)$. Therefore, $\frac{1}{c(k+1)} \sn(G) \leq \frac{1}{c(k+1)} \gon(G) \leq A'(G) \leq \sn(G) \leq \gon(G)$ and $A'$ approximates both $\sn(G)$ and $\gon(G)$ to a factor of $(k+1)c$. 
\end{proof}

\begin{corollary}
    Let $G$ be a graph with $\alpha(G) \leq \frac{n}{k+1}$ for some $ k\in \mathbb{N}$ and $\lambda_{k+1}(G) \geq n - \alpha_k(G)$, then both $\sn(G)$ and $\gon(G)$ can be $k(k+1)$ approximated in polynomial time. 
\end{corollary}
\begin{proof} 
    The claim follows by setting $c=k$ in Proposition \ref{kc-approx}
\end{proof}

It is important to note that the added condition of $\lambda_{k+1}(G) \geq n - \alpha_k(G)$ is somewhat restrictive. It does, however, expand the graphs whose scramble number and gonality can be approximated to a constant factor further than those in Theorem \ref{gon-sn-approximation}.

% \begin{proposition}
%    Let $G$ be a graph with \_. Then, $\sn(G)$ and $\gon(G)$ can be $c$-approximated in polynomial time.  
% \end{proposition}

\section{Vertex Congestion and Screewidth} \label{congestion-section}

In light of Corollary \ref{corollary:minor_closed_bounded_deg}, it is natural to ask whether there are familiar families of graphs with low scramble number.  For instance, it is well-known that planar graphs have treewidth at most $O(\sqrt{n})$; what can we say for scramble number of planar graphs?  It turns out that we can study this question using the idea of \emph{vertex congestion} of a graph.  We remark that congestion has appeared with multiple characterizations in the literature; we use the definition from \cite{embedding-graphs-in-trees} as studied in \cite{treewidth-of-line-graphs}.

We define an \emph{embedding} to be an injection $\pi$ from the vertices of $G$ to the leaves of a sub-cubic tree $T$. If $vw\in E(G)$, let $P_{vw}$ denote the path in $T$ from $\pi(v)$ to $\pi(w)$ in $T$.  The \emph{vertex congestion} of $\pi$ is the maximum over all vertices $u\in V(T)$ of the number of paths $P_{vw}$ passing through $u$; that is, the congestion equals
\[\max_u\in V(T)|\{vw\in E(G)\,|\, u\in P_{vw} \}|.\]
Finally, the \emph{vertex congestion} of $G$, denoted $\vcon(G)$, is the minimum vertex congestion over all embeddings $\pi$ of $G$.

In \cite{treewidth-of-line-graphs}, vertex congestion is characterized in terms of the treewidth of the \emph{line graph of $G$}, denoted $L(G)$; this is the graph  with vertex set $E(G)$ where two vertices are adjacent precisely when their corresponding edges are incident in $G$.
\begin{theorem} \label{theorem:vcon_tw} \cite[Theorem 2.4]{treewidth-of-line-graphs} 
For any graph $G$, $\vcon(G)=\tw(L(G))+1$.
\end{theorem}

To connect this to the treewidth of $G$, we recall the following result, presented in \cite{treewidth-of-line-graphs} but also following from results in \cite{atserias}, \cite{embedding-graphs-in-trees}, and \cite{multicuts}.

\begin{theorem}
    For a graph $G$ with maximum degree $\Delta(G)$, we have 
    \[\tw(L(G))\leq (\tw(G)+1)\Delta(G)-1\]
\end{theorem}

We now connect vertex congestion to screewidth.

\begin{theorem}\label{theorem:scw_vcon}
 For any graph $G$ with $|V(G)|\geq 3$, we have that $\scw(G)\leq \vcon(G)$.
\end{theorem}

\begin{proof}
    Let $\pi$ be an embedding of $G$ into a subcubic tree $T$ with the minimum possible vertex congestion.   Note that $\pi$ yields a tree-cut decomposition $\mathcal{T}=(T,\pi)$, since $T$ is a tree and $\pi$ is a map from $V(G)$ to $V(T)$.
    
    We now consider the width of $\mathcal{T}$.  This equals the largest value among the following sets of numbers:
    \begin{itemize}
        \item The contribution of a leaf node.  This is maximized at $1$, since $\pi$ is an injection.
        \item The contribution of a non-leaf node (at least one exists since $|V(G)|\geq 3$, implying $T$ has at least $3$ leaves).  Non-leaf nodes $u$ have no vertices assigned to them by our choice of $\pi$, so the contribution of a non-leaf node equals the number of tunneling edges.  These are precisely the edges $vw\in E(G)$ such that $P_{vw}\in u$.  Thus the maximum contribution of a non-leaf node equals the vertex congestion of $\pi$.
        \item  The contribution of a link.  Any link $l$ is adjacent to at least one non-leaf node $u$, and the edges passing through $l$ are a subset of the tunneling edges of $u$ since $u$ has no vertices assigned to it.  Thus for any link, there exists a non-leaf node with at least as large of a contribution. 
    \end{itemize}

Taking the maximum of all these numbers, we find that $w(\mathcal{T})$ is equal to the vertex congestion of $\pi$, and thus to $\vcon(G)$.  We conclude that
\[\scw(G)\leq w(\mathcal{T})= \vcon(G).\]
\end{proof}

Combining these inequalities with the fact that $\sn(G)\leq \scw(G)$ \cite{screewidth-og}, we immediately obtain Theorem \ref{theorem:sn_at_most_tw_times_delta}: for any graph $G$ of maximum degree $\Delta(G)$, we have $\sn(G)\leq (\tw(G)+1)\Delta(G)-1$.

To apply this theorem to find graphs families with small scramble number, we recall the following result.

\begin{lemma} \label{lemma:tw_no_kr}
\cite{a_separator_theorem_nonplanar} A graph on $n$ vertices without a $K_r$ minor has treewidth at most $r^{1.5}\sqrt{n}$.
\end{lemma}

\begin{corollary}\label{corollary:no_kr_minor_sn_sqrt}
Let $G$ be in a family of graphs with no $K_r$ minor and with bounded maximum degree.  Then $\sn(G) = O(\sqrt{n})$.
\end{corollary}

\begin{proof}
Theorem \ref{theorem:sn_at_most_tw_times_delta} gives us that \[\sn(G)\leq (\tw(G)+1)\Delta(G).\]
Since $\Delta(G)$ is assumed to be bounded and we have $\tw(G)=O(\sqrt{n})$ by Lemma \ref{lemma:tw_no_kr}, our claim follows.
\end{proof}

Since minor closed families of graphs with bounded maximum degree satisfy our assumptions, this subsumes Theorem \ref{corollary:minor_closed_bounded_deg}.  This corollary also implies, for instance, that planar graphs of bounded degree have $\sn(G)=O(\sqrt{n})$ since planar graphs have no $K_5$ minor.  The same argument holds if we replace ``planar graphs'' with ``graph of genus\footnote{Here by genus we mean the minimum genus of a surface onto which the graph may be embedded.  The literature on chip-firing games on graphs often uses genus to alternately mean the first Betti number of the graph.} at most $g$'' for any fixed $g$, since such graphs have at least one forbidden $K_r$ minor.

%To return to our question on planar graphs, we recall the following standard result. It follows for instance from the fact that no planar graph has $K_5$ as a minor combined with \cite{a_separator_theorem_nonplanar}, where it was proved that a graph on $n$ vertices without a $K_r$ minor has treewidth at most $r^{1.5}\sqrt{n}$.

%\begin{lemma}
%    Planar graphs on $n$ vertices have $\tw(G) = O(\sqrt{n})$.
%\end{lemma}

%This allows us to prove the following.

%\begin{corollary}
%Planar graphs of bounded degree have scramble number $\sn(G) = O(\sqrt{n})$.
%\end{corollary}

%\begin{proof}
%We have that \[\sn(G)\leq (\tw(G)+1)\Delta(G)-1.\]
%Since $\Delta(G)$ is assumed to be bounded and we have $\tw(G)=O(\sqrt{n})$, our claim follows.
%\end{proof}

Our work in this section also gives a quick proof of the following result, which also appears in \cite[Proposition 2.3]{treewidth-of-line-graphs}.
\begin{corollary} For any graph $G$, we have
    $\tw(L(G)) \geq \tw(G) - 1$.
\end{corollary}
\begin{proof}
We know that $\tw(G)\leq \sn(G)$ by \cite[Theorem 1.1]{new_lower_bound}.  The claim immediately follows from Theorems \ref{theorem:vcon_tw} and \ref{theorem:scw_vcon} and the fact that $\sn(G)\leq \scw(G)$.
\end{proof}

We close with the following open question.
\begin{question}
    Do simple planar graphs $G$ on $n$ vertices satisfy $\sn(G)=O(\sqrt{n})$?
\end{question}
The answer is ``no'' if we drop the simple assumption, since for each $n$ there exists a planar multigraph on $n$ vertices with scramble number equal to $n$.  Consider, for instance, a multigraph constructed by duplicating edges in a path on $n$ vertices until each edge is repeated $n$ times; the vertegg scramble has order $\min\{|V(G)|,\lambda(G)\}=\min\{n,n\}=n$, and any scramble $\mathcal{S}$ on a graph has order at most $h(\mathcal{S})\leq n$.  Although we could subdivide the edges of such a graph to make it simple without changing its scramble number, the number of vertices would grow quadratically, so this does not lead to an example of faster-than-square-root growth of scramble number for simple planar graphs.

% our result that carton number can be exponential
\bibliographystyle{plain}

\end{document}